\documentclass[11pt]{amsart}

\usepackage{amssymb}		
\usepackage{amsthm}		
\usepackage{enumerate}		
\usepackage{graphicx}		
\usepackage{mathrsfs}		
\usepackage[pdftex,pagebackref,letterpaper=true,colorlinks=true,pdfpagemode=none,
urlcolor=black,linkcolor=black,citecolor=black,pdfstartview=FitH]{hyperref}

\usepackage{amsmath}
\usepackage{amsfonts}
\usepackage{amssymb}
\usepackage[mathcal]{euscript}
\usepackage{pb-diagram}

\newcommand{\twiddle}{\raisebox{1pt}{\scalebox{.75}{$\mathord{\sim}$}}}

\font\logic=msam10 at 12pt
\newcommand{\forces}{\mathrel{\mbox{\logic\char'015}}}

\newcommand{\pmax}{\mathbb{P}_{\mathrm{max}}}
\newcommand{\Pmax}{\pmax}

\newcommand{\res}{\mathop{|}}
\newcommand{\rst}{\restriction}
\newcommand{\idm}{\mathsf{id}}

\newcommand{\Lp}{\mathsf{Lp}}
\newcommand{\otp}{\mathsf{otp}}
\newcommand{\eS}{\mathcal{S}}
\newcommand{\rng}{\mathrm{rng}}
\newcommand{\Le}{\mathrm{L}[E]}
\newcommand{\ult}{\mathsf{Ult}}

\newcommand{\ADR}{\mathsf{AD}_{\mathbb{R}}}
\newcommand{\HOD}{\mathrm{HOD}}
\newcommand{\Unif}{\mathsf{Uniformization}}
\newcommand{\hod}{\HOD}
\newcommand{\V}{\mathrm{V}}
\newcommand{\mbfV}{\V}

\newcommand{\PFA}{\mathsf{PFA}}
\newcommand{\ZF}{\mathsf{ZF}}

\newcommand{\MM}{\mathsf{MM}}
\newcommand{\AD}{\mathsf{AD}}
\newcommand{\DC}{\mathsf{DC}}
\newcommand{\SRP}{\mathsf{SRP}}

\newcommand{\Add}{\mathrm{Add}}
\DeclareMathOperator{\cof}{cf}
\DeclareMathOperator{\cf}{\cof}

\newcommand{\R}{\mathbb R}
\DeclareMathOperator{\dom}{dom}

\theoremstyle{definition}
\newtheorem{definition}{Definition}[section]

\newtheorem{question}[definition]{Question}

\newtheorem{remark}[definition]{Remark}

\theoremstyle{plain}
\newtheorem{theorem}[definition]{Theorem}
\newtheorem*{theorem*}{Theorem}
\newtheorem{lemma}[definition]{Lemma}
\newtheorem{corollary}[definition]{Corollary}

\newtheorem{claim}[definition]{Claim}
\newtheorem{subclaim}[definition]{Subclaim}
\newtheorem*{claim*}{Claim}

\begin{document}


\begin{abstract}
By forcing with $\pmax$ over strong models of determinacy, we obtain models where different square
principles at $\omega_2$ and $\omega_3$ fail. In particular, we obtain a model of
$2^{\aleph_0}=2^{\aleph_1}=\aleph_2+\lnot\square(\omega_2)+\lnot\square(\omega_3)$.
\end{abstract}

\author[Caicedo]{Andr\'{e}s Eduardo Caicedo$^*{}^\dagger$}
\address{
  Andr\'es Eduardo Caicedo \\
  Boise State University \\
  Department of Mathematics \\
  1910 University Drive \\
  Boise, ID 83725-1555 \\
  USA
}
\email{caicedo@math.boisestate.edu}
\urladdr{http://math.boisestate.edu/{\twiddle}caicedo}
\curraddr{Mathematical Reviews \\
  416 Fourth Street \\ 
  Ann Arbor, MI 48103-4820 \\ 
  USA
}
\email{aec@ams.org}
\urladdr{http://www-personal.umich.edu/{\twiddle}caicedo/}
\thanks{$^*$Supported in part by AIM through a SQuaREs project.}
\thanks{$\dagger$The first author was also supported in part by NSF Grant
  DMS-0801189.}

\author[Larson]{Paul Larson$^*{}^\ddagger$}
\address{
  Paul Larson \\
  Department of Mathematics \\
  Miami University \\
  Oxford, OH 45056 \\
  USA
}
\email{larsonpb@miamioh.edu}
\thanks{$\ddagger$The second author was also supported in part by NSF Grants
  DMS-0801009 and DMS-1201494.}
\urladdr{http://www.users.miamioh.edu/larsonpb/}

\author[Sargsyan]{Grigor Sargsyan$^*{}^\star$}
\address{
  Grigor Sargsyan \\
  Department of Mathematics \\
  Rutgers University \\
  Hill Center for the Mathematical Sciences, 110 Frelinghuysen Rd. \\
  Piscataway, NJ 08854-8019 \\
  USA
}
\email{grigor@math.rutgers.edu}
\thanks{$\star$The third author was also supported in part by NSF Grants DMS-0902628, DMS-1201348 and DMS-1352034.}
\urladdr{http://grigorsargis.weebly.com/}

\author[Schindler]{Ralf Schindler$^*$}
\address{
  Ralf Schindler \\
  Institut f\"{u}r mathematische Logik und Grund\-lagen\-for\-sch\-ung \\
  Fachbereich Mathematik und Informatik \\
  Universit\"{a}t M\"{u}nster \\
  Einsteinstra{\ss}e 62 \\
  48149 M\"{u}nster \\
  Germany
}
\email{rds@math.uni-muenster.de}
\thanks{The fourth author gratefully acknowledges support from the SFB 878
of the Deutsche Forschungsgemeinschaft (DFG)}
\urladdr{http://wwwmath.uni-muenster.de/logik/Personen/rds/}

\author[Steel]{John Steel$^*{}^\diamond$}
\address{
  John Steel \\
  Department of Mathematics \\
  University of California at Ber\-ke\-ley \\
  Berkeley, CA 94720 \\
  USA
}
\email{steel@math.berkeley.edu}
\thanks{$\diamond$The fifth author was also supported in part by NSF Grant DMS-0855692.}
\urladdr{http://math.berkeley.edu/{\twiddle}steel/}

\author[Zeman]{Martin Zeman$^*$}
\address{
  Martin Zeman \\
  Department of Mathematics \\
  University of California at Irvine \\
  Irvine, CA 92697 \\
  USA
}
\email{mzeman@math.uci.edu}
\urladdr{http://math.uci.edu/{\twiddle}mzeman/}



\title{Square principles in $\pmax$ extensions}

\maketitle

\tableofcontents


\section{Introduction} \label{section:introduction}

The forcing notion $\pmax$ was introduced by W. Hugh Woodin in the early 1990s, see Woodin
\cite{W}. When applied to models of the Axiom of Determinacy, it achieves a number of effects not
known to be obtainable by forcing over models of ZFC.

Recall that $\ADR$ asserts the determinacy of all length $\omega$ perfect information two player
games where the players alternate playing real numbers, and $\Theta$ denotes the least ordinal that is
not a surjective image of the reals. As usual, $\MM$ denotes the maximal forcing axiom, Martin's
Maximum. By $\MM(\mathfrak{c})$ we denote its restriction to partial orders of size at most continuum.
For the strengthening $\MM^{++}(\mathfrak{c})$ of this latter principle, see Woodin
\cite[Definition 2.47]{W}.

Woodin \cite[Theorem 9.39]{W} shows that when $\pmax$ is applied to a model of $\ADR\,+$ ``$\Theta$
is regular", the resulting extension satisfies $\MM^{++}(\mathfrak{c})$. A natural question is to what
extent one can extend this result to partial orders of size $\mathfrak{c}^{+}$. For many partial
orders of this size, obtaining the corresponding forcing axiom from a determinacy hypothesis should
greatly reduce the known upper bound for its large cardinal consistency strength. Moreover, using the
Core Model Induction, a method pioneered by Woodin (see Schindler-Steel \cite{SchinSt} and
Sargsyan \cite{S}), one can find lower bounds for the consistency strength of
$\MM^{++}(\mathfrak{c}^{+})$ and its consequences, which leads to the possibility of proving
equiconsistencies.

In this paper we apply $\pmax$ to theories stronger than
$\ADR \,+$ ``$\Theta$ is regular" and obtain some consequences of $\MM^{++}(\mathfrak c^+)$
on the extent of square principles, as introduced by Ronald B. Jensen \cite{J}. We recall the definitions:

\begin{definition}\label{squaresubkappa}
Given a cardinal $\kappa$, the principle $\square_{\kappa}$ asserts the existence of a sequence
$\langle C_{\alpha} \mid \alpha < \kappa^{+} \rangle$ such that for each $\alpha < \kappa^{+}$,
\begin{enumerate}
 \item
  $C_{\alpha}$ is club in $\alpha$;
 \item\label{subcoherence}
  for each limit point $\beta$ of $C_{\alpha}$, $C_{\beta} = C_{\alpha} \cap \beta$;
 \item
  the order type of each $C_\alpha$ is at most $\kappa$.
\end{enumerate}
\end{definition}

\begin{definition}\label{squareparendef}
Given an ordinal $\gamma$, the principle $\square(\gamma)$ asserts the existence of a sequence
$\langle C_{\alpha} \mid \alpha <\gamma \rangle$ satisfying the following conditions:
\begin{enumerate}
 \item
  For each $\alpha < \gamma$,
  \begin{itemize}
    \item
     $C_{\alpha}$ is club in $\alpha$;
    \item
      for each limit point $\beta$ of $C_{\alpha}$, $C_{\beta} = C_{\alpha} \cap \beta$.
  \end{itemize}
 \item
  There is no {\em thread} through the sequence, that is, there is no club $E\subseteq \gamma$
  such that $C_{\alpha} = E \cap \alpha$ for each limit point $\alpha$ of $E$.
\end{enumerate}
\end{definition}

We refer to sequences witnessing these principles as $\square_{\kappa}$-sequences or
$\square(\gamma)$-sequences, respectively. A sequence satisfying the first part of Definition \ref{squareparendef} (but possibly 
having a thread) is called a \emph{coherent sequence}. Note that $\square_{\kappa}$ implies
$\square(\kappa^{+})$, and that $\square_\omega$ is true.

\begin{remark} \label{rem:james}
Suppose that $\kappa$ is uncountable. A key distinction between these principles is that
$\square_{\kappa}$ persists to outer models that agree about $\kappa^{+}$, while
$\square(\kappa^{+})$ need not. This seems to be folklore; since we could not locate an argument in the
literature, we sketch one below.

For example, consider the poset ${\mathbb P}$ that attempts to add a $\square(\kappa^+)$-sequence
with initial segments. Note that ${\mathbb P}$ is $(\kappa+1)$-strategically closed.

Let $G$ be $\V$-generic for ${\mathbb P}$, and assume for the moment that the generic sequence
added by ${\mathbb P}$ is indeed a $\square(\kappa^+)$-sequence, say  $\langle C_\alpha\mid
\alpha<\kappa^+\rangle$. Then one can thread it by further forcing over $\V[G]$ with the poset
${\mathbb Q}$ whose conditions are closed bounded subsets $c$ of $\kappa^+$ such that $\max(c)$
is a limit point of $c$ and, for every $\alpha\in\lim(c)$, we have $c\cap\alpha=C_\alpha$.

This threading does not collapse $\kappa^+$, because ${\mathbb Q}$ is $\kappa^+$-distributive
in $\V[G]$, by a standard argument. In fact, the forcing ${\mathbb P}*\dot{\mathbb Q}$ has
a $\kappa^+$-closed dense set consisting of conditions of the form $(p,\dot{q})$ where $p$ decides the
value $\dot{q}$ to be $p(\alpha)$, for $\alpha$ the largest ordinal in $\dom(p)$. This can also
be verified by a standard density argument.

Assume now that $\PFA$ holds in $\V$, so $\square(\kappa^+)$ fails (as does any $\square(\gamma)$
for $\cf(\gamma)>\omega_1$, by Todorcevic \cite{T1}). Since $\PFA$ is preserved by
$\omega_2$-closed forcing, by K\"onig-Yoshinobu \cite{KY}, it holds in the extension by ${\mathbb P}*
\dot{\mathbb Q}$. (One could argue similarly starting from a universe where $\kappa$ is indestructibly
supercompact.)

It remains to argue that the sequence $\vec C$ added by ${\mathbb P}$ is a
$\square(\kappa^+)$-sequence. The (standard) argument verifying this was suggested by James
Cummings, and simplifies our original approach, where a more elaborate poset than ${\mathbb P}$ was
being used. Assume instead that the generic sequence is threadable, and let $\dot{c}$ be a name for a
thread. Now inductively construct a descending sequence of conditions $p_n$,  and an increasing
sequence of ordinals $\gamma_n$, for $n\in\omega$, such that, letting $\alpha_n$ be the length of
$p_n$, we have:
\begin{enumerate}
\item
$\alpha_n<\gamma_n<\alpha_{n+1}$,
\item
$p_{n+1}\forces \gamma_n\in \dot{c}$, and
\item
$p_{n+1}$ determines the value of $\dot{c}\cap\alpha_n$.
\end{enumerate}
Let $\gamma=\sup_n\gamma_n=\sup_n\alpha_n$, and let $p'$ be the union of all $p_n$. Then $p'$ is
not a condition, but can be made into one, call it $p^*$, by adding at $\gamma$ as the value
$p^*(\gamma)$, some cofinal subset of $\gamma$ of order type $\omega$ that is distinct from $\dot{c}
\cap\gamma$.

Then $p^*$ forces that the $\gamma$-th member of $\vec C$ is different from $\dot{c}\cap\gamma$.
But $\gamma$ is forced by $p^*$ to be a limit point of both $\dot{c}$ and the $\gamma$-th member of
$\vec C$, and therefore $p^*$ forces that $\dot{c}$ is not a thread through $\vec C$.

Viewing this as a density argument, we see that densely many conditions force that $\dot{c}$ is not a
thread thorough $\vec C$. Thus, $\vec C$ is a $\square(\kappa^+)$-sequence in $\V[G]$, as we wanted.

This shows that neither $\square(\kappa^+)$, nor its negation, is upward absolute to
models that agree on $\kappa^+$. See also the discussion on \emph{terminal square} in Schimmerling
\cite[\S 6]{Schim}.
\end{remark}

\begin{question}
Assuming that $\square(\kappa^+)$ fails, can it be made to hold by $\kappa^+$-closed forcing? This
seems unlikely, though we do not see a proof at the moment. If the answer is no, then the argument
above can be simplified, as there is no need to assume $\PFA$ or any such hypothesis on the
background universe.
\end{question}

Via work of Stevo Todorcevic \cite{T1, T2}, it is known that
$\MM^{++}(\mathfrak{c})$ implies $2^{\aleph_1}=\aleph_2+\neg\square(\omega_{2})$, and
$\MM^{++}(\mathfrak{c}^{+})$ implies $\neg\square(\omega_{3})$. Through work of Ernest
Schimmerling \cite{Schim} and Steel (via the Core Model Induction, see Schindler-Steel \cite{SchinSt}) it
is known that the following statement implies that the Axiom of Determinacy holds in the inner model
$\mathrm{L}(\mathbb{R})$:
\begin{equation} \label{equation:1}
 \neg\square(\omega_{2}) + \neg\square_{\omega_{2}} + 2^{\aleph_{1}} = \aleph_{2}.
\end{equation}

Theorem \ref{hyplowerbound} below shows that if $\ADR\,+\text{\rm``$\Theta$ is regular"}$ 
holds, and there is no
$\Gamma \subseteq \mathcal{P}(\mathbb{R})$ such that
$\mathrm{L}(\Gamma,\mathbb{R})\models \text{\rm``$\Theta$ is Mahlo in $\HOD$"}$, then a
weak form of $\square_{\omega_{2}}$ (restricted to a club subset of $\Theta$, which becomes $\omega_{3}$ in the $\pmax$ extension) 
holds in $\HOD$ (the inner model of all hereditarily ordinal definable sets, see Jech \cite{Jech}).
The argument is in essence a standard adaptation of the usual proof of square
principles in fine structural models; the $\hod$ analysis of Sargsyan \cite{S} makes this adaptation
possible.
As shown in Corollary \ref{hyplowerboundcor}, this gives a lower bound for the hypotheses needed for the threadability results in this paper. 

The axiom $\AD^{+}$ is an ostensible strengthening of $\AD$, introduced by Woodin (we refer the reader to page 611 of \cite{W} for the definition). 
It is an open question whether the 
two statements are equivalent.
It is also an open question whether $\AD^{+}$ follows
from $\ADR$; Woodin has shown that it does follow from $\ADR$ + $\DC$. 

Woodin's analysis of $\pmax$ is carried out abstractly in the context of $\AD^{+}$. Woodin has shown that the theory $\AD^{+}$ + $\Unif$ is equivalent to the theory 
$\AD^{+} + \ADR$ (see Theorem 9.22 or Theorem 9.24 of \cite{W}); we use this theory for 
several of our results.\footnote{$\Unif$ is the statment that if $A \subseteq \mathbb{R} \times \mathbb{R}$ is such that each 
vertical cross-section of $A$ is nonempty, then $A$ contains the graph of a function with domain $\mathbb{R}$. $\Unif$ follows easily 
from $\ADR$.}
A hypothesis weaker than $\AD^{+}$ + $\Unif$ + ``$\Theta$ is Mahlo in $\HOD$'' suffices to make $\square_{\omega_2}$
fail in the $\Pmax$ extension; see Theorem \ref{theorem:gone}.
However, this extension will not be a model of Choice: 
a standard argument (see Remark \ref{nowoname}) shows that if $\AD$ and $\Unif$ hold, then forcing with $\pmax$ cannot wellorder $\mathcal{P}(\mathbb{R})$.
Theorem \ref{theorem:weaklycompact} gives a result on the failure of $\square(\Theta)$ in determinacy models. 

The hypothesis that we use to force the negation of $\square_{\omega_2}$ in a model of Choice is in the end just slightly
stronger than the assumption of Theorem \ref{hyplowerbound} (see Theorem \ref{theorem:sqcofoo}, and the discussion before) :
\begin{quote}
$\AD^{+}$ + $\Unif$ + $\V = \mathrm{L}(\mathcal{P}(\mathbb{R}))$ + ``$\{\kappa \mid \kappa$ is
regular in $\HOD$, is a member of the Solovay sequence, and has cofinality $\omega_1\}$ is
stationary in $\Theta$".
\end{quote}
We show from this hypothesis that $\square_{\omega_{2}}$ fails in the extension given by
$\pmax$ followed by a natural forcing ($\Add(\omega_{3}, 1)$) well-ordering the power set of the reals. That  $2^{\aleph_1}=
\aleph_2$ and $\lnot\square(\omega_2)$ also hold follows from Woodin's work. Similarly, a stronger
hypothesis allows us to conclude that even $\square(\omega_{3})$ fails
in the final extension, see Theorem \ref{theorem:vanilla}.
 
The determinacy
hypotheses we use are all weaker in consistency strength than a Woodin cardinal that is limit of Woodin
cardinals \cite{S, S15a, S15b, SarTra}. This puts them within the region suitable to be reached from current techniques by a Core
Model Induction. Moreover, these hypotheses are much weaker than the previously known upper
bounds on the strength of (\ref{equation:1}) and similar theories.

Prior to our work, two methods were known to show the consistency of (\ref{equation:1}): It is a
consequence of $\PFA({\mathfrak c}^+)$, the restriction of the proper forcing axiom to partial orders of
size ${\mathfrak c}^+$ (see Todorcevic \cite{T1}), and it can be forced directly from the existence of a
{\em quasicompact} cardinal. Quasicompactness was introduced by Jensen, see Cummings
\cite{cummings} and Jensen \cite{jensen}; Sean Cox (unpublished), and possibly others, observed that
the classical argument from James E. Baumgartner \cite{Baum} obtaining the consistency of ``every
stationary subset of $\omega_2\cap{\rm cof}(\omega)$ reflects'' from weak compactness adapts
straightforwardly to this setting.

We expect that the $\HOD$ analysis (see Sargsyan \cite{S}) should allow us to extend the Core Model
Induction to establish the precise consistency strength of (\ref{equation:1}). The question of whether it is
possible to obtain  $\MM^{++}(\mathfrak c^+)$ or even $\PFA({\mathfrak c}^+)$ in a $\pmax$ extension
of some determinacy model remains open.

Since forcing axioms are connected with {\em failures} of square principles, we want to suggest some
notation to refer to these negations in a positive way, highlighting their {\em compactness} character,
and solving the slight notational inconvenience that refers to the square principle at a cardinal
successor $\kappa^+$ as $\square_\kappa$, as if it were a property of its predecessor.

\begin{definition} \label{def:threadable}
Let $\gamma$ be an ordinal. We say that \textbf{\em$\gamma$ is threadable} if and only if
$\square(\gamma)$ fails.

If $\gamma=\lambda^+$ is a successor cardinal, we say that \textbf{\em
$\gamma$ is square inaccessible} if and only if $\square_\lambda$ fails.
\end{definition}

For general background on descriptive set theory, we refer to Kechris \cite{Kechris} and Moschovakis
\cite{Moschovakis}; the latter is also a good reference for basic determinacy results. In addition, for
determinacy and Woodin's $\AD^+$ theory, we also refer to Woodin \cite{W}, Jackson \cite{Jackson},
Caicedo-Ketchersid \cite{CK}, Ketchersid \cite{K}, and references therein. Basic knowledge of
determinacy will be assumed in what follows.

We also assume some ease with $\Pmax$ arguments, although the properties of $\Pmax$ that we
require could be isolated and treated as black boxes. We refer to Woodin \cite{W} and Larson
\cite{Larson} for background.

\subsection{Acknowledgements}

The first, second, third, and fifth authors were supported in part by NSF Grants  DMS-0801189,
DMS-0801009 and DMS-1201494, DMS-0902628, and DMS-0855692, respectively.

This work started during a week-long meeting at the American Institute of Mathematics during May
16--20, 2011, as part of the SQuaRE (Structured Quartet Research Ensemble) project ``Aspects of
descriptive inner model theory", and also incorporates results obtained during the second and third such meetings
during April 16--20, 2012 and May 20--24, 2013. We are grateful to the Institute for the opportunity and support.

We also want to thank James Cummings for a suggestion that we incorporated in the discussion in
Remark \ref{rem:james}, and the referee for a careful reading and several useful suggestions.

\section{From $\HOD$ to $\HOD_{\mathcal{P}(\mathbb{R})}$} \label{section:hypotheses}

Some of our results use hypotheses on inner models of the form
$\HOD_{\mathcal{P}_{\kappa}(\mathbb{R})}$,
where $\mathcal{P}_{\kappa}(\mathbb{R})$ denotes the collection of sets of reals of Wadge rank
less than $\kappa$. 
We show in this section how some of these hypotheses follow from statements about
$\HOD$. With the possible exception of Theorem \ref{getmeas}, none of the results in this section is
new. The key technical tool is given by Lemma \ref{voplem}.

We refer the reader to Kechris \cite{Kechris} for background on the Wadge
hierarchy. If $A$ is a set of reals, we let $|A|_W$ denote its Wadge rank.
Recall that $\Theta$ is the least ordinal that is not a surjective image of the reals, and that the
\emph{Solovay sequence}, introduced by Robert M. Solovay \cite{Solovay}, is the unique increasing
sequence of ordinals $\langle \theta_{\alpha} \mid \alpha \leq \gamma \rangle$ such that
\begin{itemize}
\item
$\theta_{0}$ is the least ordinal that is not the surjective image of the reals by an ordinal
definable function;
\item
for each $\alpha < \gamma$, $\theta_{\alpha + 1}$ is the least ordinal that is not the surjective
image of the reals by a function definable from an ordinal and a set of reals of Wadge rank
$\theta_{\alpha}$;
\item
for each limit ordinal $\beta \leq \gamma$, $\theta_{\beta} = \sup\{\theta_{\alpha} \mid \alpha < \beta\}$;
and
\item
$\theta_{\gamma} = \Theta$.
\end{itemize}

\begin{remark}\label{Solovayremark}
The proof of Solovay \cite[Lemma 0.2]{Solovay} shows that, under $\AD$,
whenever $\gamma$ is an ordinal, and $\phi \colon \mathbb{R} \to \gamma$ is a surjection, there exists
a set of reals of Wadge rank $\gamma$ definable from $\phi$. From this it follows that if $|F|_W =
\theta_\alpha < \Theta$, then every set of reals of Wadge rank less than
$\theta_{\alpha + 1}$ is definable from $F$, a real and an ordinal. In turn, it follows from this that there is
no surjection from $\mathbb{R}$ to $\theta_{\alpha + 1}$ ordinal definable from any set of reals of
Wadge rank less than $\theta_{\alpha + 1}$. It follows moreover that for each $\alpha \leq \gamma$ as in the definition of the Solovay sequence, 
$\mathcal{P}_{\theta_{\alpha}}(\mathbb{R}) = \mathcal{P}(\mathbb{R}) \cap \HOD_{\mathcal{P}_{\theta_{\alpha}}(\mathbb{R})}$ and $\theta_{\alpha}$ is the $\Theta$ of
$\HOD_{\mathcal{P}_{\theta_{\alpha}}(\mathbb{R})}$.
\end{remark}

The following is due to Woodin, building on work of Petr Vop\v enka (see Jech
\cite[Theorem 15.46]{Jech}):

\begin{lemma}[$\mathsf{ZF}$]\label{voplem} For each ordinal $\xi$ there exists a complete Boolean algebra
$\mathbb{B}$ in $\HOD$ such that  for each $E \subseteq \xi$ there is a $\HOD$-generic filter
$H \subseteq \mathbb{B}$ with
 $$ \HOD_{E} = \HOD[H]. $$
Furthermore, if $\AD$ holds, $\xi < \Theta$ and $\theta$ is the least member of the Solovay sequence
greater than $\xi$, then $\mathbb{B}$ can be taken to have cardinality at most $\theta$ in $\V$.
\end{lemma}

\begin{proof}
Fix an ordinal $\gamma$ such that every ordinal definable subset of $\mathcal{P}(\xi)$ is ordinal
definable in $\V_{\gamma}$. Let $\mathbb{B}_{0}$ be the following version of the \emph{Vop\v enka
algebra}: ${\mathbb B}_0$ is the Boolean algebra consisting of all sets of the form
 $$ A_{\phi, s} = \{ x \subseteq \xi \mid \V_{\gamma} \models \phi(x,s)\}, $$
where $\phi$ is a formula and $s$ is a finite subset of $\gamma$, ordered by inclusion.

The relation $A_{\phi, s} = A_{\phi',s'}$ is ordinal definable, and there is an ordinal definable
well-ordering of the corresponding equivalence classes. Let $\eta$ be the length of this well-ordering,
let $h \colon \eta \to \mathbb{B}_{0}$ be the corresponding inverse rank function, and let
$\mathbb{B}_{1}$ be the Boolean algebra with domain $\eta$ induced by $h$. Since the relation
$A_{\phi,s} \subseteq A_{\phi',s'}$ is ordinal definable, $\mathbb{B}_{1}$ is in $\HOD$.

Given a filter $G \subseteq \mathbb{B}_{0}$, let $E(G)$ be the set of $\alpha \in \xi$ such that
$$\{ E\subseteq \xi \mid \alpha \in E\} \in G.$$
Then for any $E \subseteq \xi$, $E(\{A \in \mathbb{B}_{0} \mid E \in A\}) = E$. According to
Vop\v{e}nka's Theorem (see Caicedo-Ketchersid \cite{CK}),
\begin{enumerate}
\item\label{allin}
For every $E \subseteq \xi$, $H_{E} = h^{-1}[\{A \in \mathbb{B}_{0} \mid E \in A\}]$ is $\HOD$-generic
for $\mathbb{B}_{1}$;
\item\label{splitsend}
There exists a $\mathbb{B}_{1}$-name $\dot{E} \in \HOD$ such that  if $H \subseteq \mathbb{B}_{1}$
is $\HOD$-generic and $G = h[H]$, then $E(G) = \dot{E}_{H}$.
\end{enumerate}

Now suppose that $H \subseteq \mathbb{B}_{1}$ is $\HOD$-generic, and let $E = \dot{E}_{H}$.
Let us see first that $\HOD_{E} \subseteq \HOD[H]$. Suppose that $A$ is a set of ordinals in
$\HOD_E$, and fix an ordinal $\delta$ such that $A$ and $E$ are both subsets of $\delta$. Then there
is an ordinal definable relation $T \subseteq \delta \times \mathcal{P}(\delta)$ such that
$A = \{\zeta < \delta\mid T(\zeta,E)\}$. Define the relation $T^{*}$ on $\delta \times \mathbb{B}_{1}$ by
setting $T^*(\zeta,p)$ if and only if $T(\zeta,D)$ holds for all $D \in h(p)$. Then $T^{*} \in \HOD$, and
$A$ is in $\HOD[H]$, since $A$ is the set of $\zeta < \sup(A)$ such that there exists a $p$ in $H$ for
which $T^*(\zeta,p)$ holds.

For the other direction, note first that for any set $F \subseteq \xi$, $H_{F}$ is in $\HOD_{F}$. It follows that $H_{E}$ is in $\HOD[H]$. Moreover, 
the previous paragraph shows that there is a $\mathbb{B}_{1}$-name in $\HOD$ for $H_{\dot{E}_{H}}$. Densely many conditions
in $\mathbb{B}_{1}$ then decide whether or not the generic filter $H$ will be equal to $H_{\dot{E}_{H}}$.
However, for any condition $p \in \mathbb{B}_{1}$, if $F$ is any element of $h(p)$ then
$\dot{E}_{H_{F}} = F$, which means that $p$ cannot force the generic filter $H$ to be different from
$H_{\dot{E}_{H}}$. It follows then that $H$ is in $\HOD_{E}$ so $\HOD[H] = \HOD_{E}$, whenever
$H \subseteq \mathbb{B}_{1}$ is $\HOD$-generic and $E = \dot{E}_{H}$.

Finally, assume that $\AD$ holds, and let $\theta$ be as in the statement of the lemma. Then $\theta$ is
either $\theta_{0}$ or $\theta_{\alpha + 1}$ for some $\alpha$. Let $F$ be empty if the first case holds,
and a set of reals of Wadge rank $\theta_{\alpha}$ otherwise.

Let us see that the cardinality of $\eta$ is at most $\theta$ in $\V$. By Remark \ref{Solovayremark} and
the Moschovakis Coding Lemma (see for instance Koellner-Woodin \cite[Theorem 3.2]{KoellnerWoodin} or Kanamori \cite[Theorem 28.15]{kanamori}) there is a
surjection $\pi \colon \mathbb{R} \to \mathcal{P}(\xi)$ definable from $\xi$ and
$F$. If $A$ is an ordinal definable subset of $\mathcal{P}(\xi)$, then $\pi^{-1}[A]$ is ordinal definable
from $F$, which means that
$|\pi^{-1}[A]|_W<\theta$, which again by Remark \ref{Solovayremark} implies
that $A$ is definable from $F$,
a real and a finite subset of $\theta$. For each fixed finite $a \subseteq \theta$, the definability order on the
sets $A \subseteq \mathcal{P}(\xi)$ definable from $F$, $a$, and a real, induces a pre-well-ordering of
the reals definable from $F$ and $a$, which must then have
order type less than $\theta$. It follows from this that $\eta$ also has cardinality at most $\theta$.
\end{proof} 

By Theorem 9.10 of \cite{W}, 
if $\AD^{+}$ holds, then it holds in every inner model of $\ZF$ containing $\mathbb{R}$. 

\begin{remark}\label{boundedsetrem} Under $\AD^{+}$, every set of reals is ordinal definable from a set of ordinals, and in fact this set can be
taken to be a bounded subset of $\Theta$ (see Woodin \cite[Lemma 9.5]{W}).  
\end{remark}

Combining this fact with Lemma \ref{voplem} gives the following folklore result.

\begin{theorem}\label{regpres}
Assume that $\AD^{+}$ holds. If $\theta$ is a member of the Solovay sequence and $\theta$ is regular in
$\HOD$, then $\theta$ is regular in $\HOD_{\mathcal{P}_{\theta}(\mathbb{R})}$. For any $S \subseteq \theta$ in $\HOD$,
if $S$ is stationary in $\HOD$ then $S$ is stationary in $\HOD_{\mathcal{P}_{\theta}(\mathbb{R})}$.
\end{theorem}

\begin{proof}
We prove the first part and leave the proof of the second, which is similar, to the reader. 
If $\theta = \theta_{0}$ or $\theta$ is a successor member of the Solovay sequence, there
is a set of reals $A$ from which Wadge-cofinally many sets of reals are ordinal definable. It follows then
that $\theta$ is regular in $\HOD_{\mathcal{P}_{\theta}(\mathbb{R})}$ without any assumption
on $\HOD$, since for each function $f \colon \mathbb{R} \to \theta$ in $\HOD_{\mathcal{P}_{\theta}(\mathbb{R})}$ there is a surjection from
$\mathbb{R}$ to $\sup(f[\mathbb{R}])$ definable from $A$.

Now suppose that $\theta$ is a limit in the Solovay sequence, and that $f\colon\alpha\to\theta$ is a
cofinal function in $\HOD_{\mathcal{P}_{\theta}(\mathbb{R})}$, for some
$\alpha<  \theta$. Then $f$ is ordinal definable from some set of reals in $\mathcal{P}_{\theta}
(\mathbb{R})$ that itself is ordinal
definable from a bounded subset $A$ of
$\theta$ (see Remark \ref{boundedsetrem}).

Pick $\theta_\xi<  \theta$ such that $A$ is bounded in $\theta_\xi$.
By Lemma \ref{voplem}, there is a set $H$, generic over $\HOD$ via a partial order of
cardinality less than $\theta_{\xi+1}$ in $\HOD$,
and such that $\HOD_{A} \subseteq \HOD[H]$. Then $f \in \HOD[H]$. By cardinality considerations, the
regularity of $\theta$ in $\HOD$ is preserved in $\HOD[H]$, giving a contradiction.
\end{proof}

With a little more work, one gets Theorem \ref{getmeas} below. 

\begin{remark}\label{nowoname} A standard argument shows that, assuming that both $\AD$ and $\Unif$ hold, there is no set $A$ such that 
every set of reals is ordinal definable from $A$ and a real (consider the set of pairs $(x,y)$ such that $y$ is not 
ordinal definable from $A$ and $x$). It follows, for instance, that $\AD$ + $\Unif$ implies that there is 
no function from an ordinal to $\mathcal{P}(\mathbb{R})$ whose range is Wadge-cofinal, and also that the Solovay sequence has 
limit length (applying Remark \ref{Solovayremark}). Similarly, since $\pmax \subseteq H(\aleph_{1})$, 
$\AD$ + $\Unif$ implies that there is no $\pmax$-name for a wellordering of $\mathcal{P}(\mathbb{R})$. 
\end{remark} 

By Theorem 9.24 of \cite{W} 
(modulo the fact that $\ADR$ reflects to the inner model $\mathrm{L}(\mathcal{P}(\mathbb{R}))$), if $\AD^{+}$ holds and the Solovay sequence 
has limit length then $\ADR$ (and thus $\AD$ + $\Unif$) holds. 


Given a cardinal $\theta$, a filter on $\theta$ is $\theta$-\emph{complete} if it is closed under
intersections of cardinality less than $\theta$, and $\mathbb{R}$-\emph{complete} if it is closed under
intersections indexed by $\mathbb{R}$.

\begin{theorem}\label{getmeas}
Assume that $\AD^{+}$ holds, and that $\theta$ is a limit on the Solovay sequence. Let $F$ be a
$\theta$-complete filter on $\theta$ in $\HOD$, and let $F'$ be the set of elements of
$\mathcal{P}(\theta) \cap \HOD_{\mathcal{P}_{\theta}(\mathbb{R})}$ containing some member of $F$.
Then the following hold:
\begin{enumerate}
\item The filter $F'$ is $\mathbb{R}$-complete in $\HOD_{\mathcal{P}_{\theta}(\mathbb{R})}$.
\item If $F$ is normal in $\HOD$, then $F'$ is normal in $\HOD_{\mathcal{P}_{\theta}
(\mathbb{R})}$.
\item If $F$ is an ultrafilter in $\HOD$, then $F'$ is an ultrafilter in $\HOD_{\mathcal{P}_{\theta}
(\mathbb{R})}$.
\end{enumerate}
Moreover, every $\theta$-complete filter on $\theta$ in $\HOD_{\mathcal{P}_{\theta}(\mathbb{R})}$ is
$\mathbb{R}$-complete.
\end{theorem}

\begin{proof}
Again by Remark \ref{boundedsetrem} (and the last line of Remark \ref{Solovayremark}), every element of $\HOD_{\mathcal{P}_{\theta}(\mathbb{R})}$ is ordinal definable from a bounded subset of $\theta$.
By Lemma \ref{voplem}, every element of $\HOD_{\mathcal{P}_{\theta}(\mathbb{R})}$ exists in a generic extension of $\HOD$
by a partial order of cardinality less than $\theta$ in $\HOD$. The second and third conclusions of the lemma
follow, as well as the fact that $F'$ is $\theta$-complete in $\HOD_{\mathcal{P}_{\theta}(\mathbb{R})}$. It
suffices then to prove the last part of the theorem.

Let $F'$ be a $\theta$-complete filter on $\theta$, and $G$ a function from $\mathbb{R}$ to $F'$, both in
$\HOD_{\mathcal{P}_{\theta}(\mathbb{R})}$. For each $\alpha < \theta$, let
$B_{\alpha} = \{ x \in \mathbb{R} \mid\alpha \in G(x)\}$. Then $\bar{B} = \langle B_{\alpha} \mid
\alpha < \theta \rangle$ is ordinal definable from a set of reals of Wadge rank less than $\theta$.
Since $\AD$ + $\Unif$ holds in $\HOD_{\mathcal{P}_{\theta}(\mathbb{R})}$ 
(see the remarks Remarks \ref{boundedsetrem} and \ref{nowoname}), and $\theta$ is the $\Theta$ of
this model (see Remark \ref{Solovayremark}), the Wadge ranks of the $B_{\alpha}$'s cannot be cofinal in 
$\theta$ (by Remark \ref{nowoname}).
On the other hand, if unboundedly many of the $B_{\alpha}$ were to be distinct sets Wadge-below some fixed set of reals, 
then one could define from this situation a pre-well-ordering of length
$\theta$, which is impossible. So $\bar{B}$ must contain fewer than $\theta$ many distinct sets.

Suppose that $B \subseteq \mathbb{R}$ is such that $\{ \alpha < \theta \mid B_{\alpha} = B \}$ is $F'$-
positive. For each $x \in \mathbb{R}$, $G(x) \in F'$, so there is an $\alpha < \theta$ for which $x \in
B_{\alpha}$ and $B_{\alpha} = B$.
It follows that $B = \mathbb{R}$. Since $F'$ is $\theta$-complete, $$\{ \alpha \mid B_{\alpha} =
\mathbb{R}\} \in F'.$$
Since $\bigcap_{x \in \mathbb{R}} G(x) = \{ \alpha \mid B_{\alpha} = \mathbb{R}\}$, we are done.
\end{proof}

\section{Square in $\pmax$ extensions of weak models of determinacy} \label{section:square}

Theorem ~\ref{hyplowerbound} below shows that under a certain minimality hypothesis on 
our determinacy models, partial $\square$-sequences of length $\Theta$ exist. 
Corollary \ref{hyplowerboundcor} then shows that $\square_{\omega_{2}}$ holds in the 
$\pmax * \Add(\Theta, 1)$-extension of a model of this hypothesis. In Section \ref{section:nochoice}
we will see that $\pmax$ alone does not necessarily add a $\square_{\omega_{2}}$-sequence in this 
context. 

\begin{theorem}\label{hyplowerbound}
Assume that 
$\ADR$ holds
and that there is no $\Gamma \subseteq \mathcal{P}(\mathbb{R})$ such that
$\mathrm{L}(\Gamma, \mathbb{R}) \models \text{\rm``$\Theta$ is Mahlo in
$\HOD$"}$.  Then in $\HOD$ there exist a closed unbounded set 
$C^*\subseteq\Theta$ and a
sequence $\langle c_\tau\mid\tau\in C^*\rangle$ such that the following hold. 
\begin{itemize}
\item[(a)] Each $c_\tau$ is a closed unbounded subset of $\tau$ and
           $c_\tau\subseteq C^*$ unless the order-type of $c_\tau$ is $\omega$. 
\item[(b)] If $\bar{\tau}$ is a limit point of $c_\tau$ then
           $c_{\bar{\tau}}=c_\tau\cap\bar{\tau}$. 
\item[(c)] $\otp(c_\tau)<\tau$. 
\end{itemize}
\end{theorem}

\begin{remark} The assumption that $\Theta$ is not Mahlo in $\mathrm{L}(\mathcal{P}(\mathbb{R}))$ is necessary for the theorem above, but
the hypothesis that there is no $\Gamma$ properly contained in $\mathcal{P}(\mathbb{R})$
such that $\mathrm{L}(\Gamma, \mathbb{R}) \models \text{\rm``$\Theta$ is Mahlo in $\HOD$"}$ may possibly be weakened. Its role is to make the $\mathrm{HOD}$ analysis from
\cite{S} applicable.
\end{remark}

The heart of the construction takes place in $\hod$ and is a straightforward
combination of standard constructions of square sequences as developed in Jensen \cite{J},
Schimmerling-Zeman \cite{chsq}, and Zeman \cite{gls}, adapted to the context of strategic extender
models as developed in Sargsyan \cite{S}. In order to stay close to the constructions in
Schimmerling-Zeman \cite{chsq} and Zeman \cite{gls}, we use fine structure notation and terminology
as in Zeman \cite{imlc}; the rest of the notation and terminology is consistent with that in Mitchell-Steel
\cite{fsit}, Steel \cite{cmip}, and Sargsyan \cite{S}.

\begin{proof}
The key technical tool is a condensation lemma for initial segments of $\hod$
which can be proved using the standard argument modified to the strategic
extender models from Sargsyan \cite{S}. Unlike the square constructions in
Schimmerling-Zeman \cite{chsq} and Zeman \cite{gls}, our situation is specific in the sense that the
initial segments of $\hod$ used for the definition of the elements of our square sequence are
never pluripotent (see Schimmerling-Zeman \cite{chsq} and Zeman \cite{gls}), that is, they do not give
rise to protomice. This makes it possible to run the construction without analysis of
extender fragments, so the construction does not differ too much from that in
$\mathrm{L}$. The reason why this is the case is a consequence of the
following corollary of our smallness assumption, see Sargsyan \cite{S}:

\begin{equation}\label{n:e:nonoverlap}
\mbox{If $\theta_\alpha<\Theta$, then $\theta_\alpha$ is not overlapped by an
extender on the $\hod$-sequence.}
\end{equation}

We now formulate the condensation lemma. Recall that if
$\theta_\alpha<\Theta$, then $\Sigma_\alpha$ is the iteration strategy for
$\hod\res\theta_\alpha$ in $\hod$.

\begin{lemma}\label{n:l:condensation}
Assume ${\mathcal N}$ is an initial segment of $\hod$. Let $\theta_\alpha<\Theta$, and let
${\mathcal M}$ be a sound strategic premouse such that
$\rho^{n+1}_{{\mathcal M}}=\theta_\alpha$. Let finally $\sigma:{\mathcal M}\to{\mathcal N}$ be a
$\Sigma^{(n)}_0$-preserving map with critical point $\theta_\alpha$, $\sigma(\theta_\alpha)=
\theta_\beta$, and $\sigma\in\hod$. Then ${\mathcal M}$ is an initial segment of $\hod$.
\end{lemma}

\begin{proof}
(Sketch.)
Since $\sigma\rst\theta_\alpha=\idm$, the model ${\mathcal M}$ agrees with $\hod$ below
$\theta_\alpha$.

By (\ref{n:e:nonoverlap}), all critical points of the iteration tree on the
${\mathcal M}$-side of the comparison of ${\mathcal M}$ against $\hod$ are strictly larger than
$\theta_\alpha$. The preservation degree of $\sigma$ guarantees that ${\mathcal M}$
is iterable when using extenders with critical points larger than $\theta_\alpha$, so
${\mathcal M}$ can be compared with $\hod$. By the theory developed in Sargsyan \cite{S},
$\hod$ wins the comparison against ${\mathcal M}$; the assumption $\sigma\in\hod$ is used
here. (We sketch the argument below, using freely notation and results form Sargsyan \cite{S}.)
But then ${\mathcal M}$ is not moved in the coiteration, as it projects to
$\theta_\alpha$, is sound, and all critical points on the ${\mathcal M}$-side are
larger than $\theta_\alpha$. This gives the result. This is an instance of a more general result from
Sargsyan \cite{S}, namely, that $\hod$ thinks that it is {\em full}: Letting
$\lambda^{\hod}$ denote the order type of the set of Woodin cardinals in $\HOD$ and their limits, if
$\alpha<\lambda^{\hod}$, and $\eta\in[\theta_\alpha,\theta_{\alpha+1})$ is a cutpoint, then any sound
$\Sigma_\alpha$-mouse ${\mathcal M}$ over $\hod\res\eta$ with $\rho({\mathcal M})=\eta$ {\em is} an
initial segment of $\hod$.

To see that $\hod$ wins the comparison, let $\Sigma$ be the strategy of ${\mathcal N}$, so $\Sigma$
respects $\Sigma_\beta$, the iteration strategy for $\hod\res\theta_\beta$ in $\hod$. By the arguments of
Sargsyan \cite{S}, all we need to check is that ${\mathcal M}$ is a
$\Sigma_\alpha$-premouse and $\Sigma^\sigma$ respects $\Sigma_\alpha$. This
follows from hull condensation, and the former is immediate.

\vspace{-3mm}
\begin{equation} \label{displayedeqn}
\begin{diagram}
\node{{\mathcal N}}
\arrow{e,t}{{{\mathcal T}^\sigma}}
\node{{\mathcal N}'}\\
\node{{\mathcal M}}
\arrow{n,l}{{\sigma}}
\arrow{e,b}{{\mathcal T}}
\node{{\mathcal M}'}
\arrow{n,r}{{\sigma'}}
\end{diagram}
\end{equation}

To see this, suppose that ${\mathcal T}$ is the tree arising on the ${\mathcal M}$-side of the
comparison, as in diagram (\ref{displayedeqn}), so ${\mathcal T}$ has critical point above
$\theta_\alpha$. Let ${\mathcal U}$ be on the sequence of ${\mathcal M}'$. We must argue that
${\mathcal U}$ is according to $\Sigma_\alpha$. We use $\sigma'$ to produce $\sigma'({\mathcal U})$,
and note that the pointwise image $\sigma'[{\mathcal U}]$ is a hull of $\sigma'({\mathcal U})$. But then
${\mathcal U}$ is also a hull of $\sigma'({\mathcal U})$, and since $\sigma\in\hod$, so is
$\sigma'\in\hod$. This allows us to invoke hull condensation, as claimed.
\end{proof}

In the following we treat the case where $\Theta=\theta_\Omega$ for some
limit ordinal $\Omega$ which is not Mahlo in $\hod$. The remaining
case, namely when $\Theta$ is of the form
$\theta_{\alpha+1}$, is discussed at the end of this section.

We modify the construction in Zeman \cite{gls} to the current context to
obtain a global square sequence below $\Theta$ in $\hod$ of the form
$\langle C_\tau\mid\tau\in C^*\rangle$ where $C^*\in\hod$ is a closed
subset of $\Theta$ fixed in advance such that
\begin{itemize}
\item $\theta_\tau=\tau$ and $\tau$ is singular in $\hod$ whenever 
      $\tau\in C^*$; 
\item $\cof(\tau)=\omega$ whenever $\tau$ is a successor point of $C^*$. 
\end{itemize}
Notice that the latter can be arranged by replacing $C^*$ with $\lim(C^*)$. 

In the rest of this section, we describe the modifications to the construction in
Zeman \cite{gls} that yield the sequence $\langle C_\tau\mid\tau\in
C^*\rangle$.
The point of our description is to separate aspects of the construction that
can be accomplished by abstract fine structural considerations from those that
are specific to strategic extender models. By ``abstract fine structural
considerations'', we mean here methods within the framework of Zeman
\cite[Chapter 1]{imlc}.  Our notation is consistent with that in Zeman \cite{gls} with two
exceptions: First, in our case $\Theta$ plays the same role the class of ordinals
plays in Zeman \cite{gls}, and second, the sets $C_\tau$ from Zeman \cite{gls} do not
correspond to the sets denoted by $C_\tau$ here.

The construction in Zeman \cite{gls} is carried out separately
on two disjoint sets $\eS^0,\eS^1$. The class $\eS^1$ consists of all those
ordinals for which the singularizing structure is a protomouse. In our case,
protomice do not arise in the construction due to (\ref{n:e:nonoverlap}), so
we have $\eS^1=\varnothing$. This greatly simplifies the situation, as the
verification that $C_\tau\subseteq\eS^i$ whenever $\tau\in\eS^i$ ($i=0,1$)
and $\otp(C_\tau)>\omega$, involved a substantial amount of work in Zeman \cite{gls}.
Thus we will refer only to the portion of the construction in Zeman \cite{gls}
that concerns the set $\eS^0$.

To each $\tau\in C^*$, we assign the singularizing level of $\hod$ for
$\tau$, which we denote by ${\mathcal N}_\tau$. We then define the auxiliary
objects for ${\mathcal N}_\tau$ exactly as in Zeman \cite{gls}:
\begin{itemize}
\item
By $\tilde{h}^k_\tau$, we denote the uniform $\Sigma^{(k-1)}_1$-Skolem function for
${\mathcal N}_\tau$; this is a partial function from $\omega\times{\mathcal N}_\tau$ into
${\mathcal N}_\tau$.
\item
We write $\tilde{h}^k_\tau(\gamma\cup\{p\})$ to denote the set of all
values $\tilde{h}^k_\tau(i,\langle\xi,p\rangle)$, where $i\in\omega$ and
$\xi<\gamma$.
\item
We let $p_\tau$ be the standard parameter of ${\mathcal N}_\tau$.
\item
We let $n_\tau$ be the complexity degree of a singularizing function for $\tau$ over
${\mathcal N}_\tau$, or equivalently the least $n$ such that
$\tilde{h}^{n+1}_\tau(\gamma\cup\{p_\tau\})$ is cofinal in $\tau$ for some
$\gamma<\tau$.
\item
$\tilde{h}_\tau=\tilde{h}^{n_\tau+1}_\tau$
\item
We let $\alpha_\tau$ be the largest $\alpha<\tau$ such that
$\tilde{h}_\tau(\alpha\cup\{p_\tau\})\cap\tau=\alpha$.
\end{itemize}

We then define the sets $B_\tau$, which are the first approximations to
$C_\tau$, analogously as in Zeman \cite{gls}. Recall that $B_\tau$ may be
bounded in $\tau$ and even empty if $\tau$ is countably cofinal, but on the
other hand $B_\tau$ will be ``almost'' coherent also at successor points.

To be precise: Recall that an ordinal $\zeta$ is in $p_\tau$ if and only if some generalized solidity
witness for $\zeta$ with respect to ${\mathcal N}_\tau$ and $p_\tau$ is an element of
${\mathcal N}_\tau$, and the standard solidity witness for $\zeta$ can be reconstructed
from any generalized solidity witness for $\zeta$ inside ${\mathcal N}_\tau$. Here, by the
\emph{standard} solidity witness, we mean the transitive
collapse of the hull $\tilde{h}^{k+1}_\tau(\zeta\cup\{p_\tau-(\zeta+1)\})$, where
$\rho^{k+1}_{{\mathcal N}_\tau}\le\zeta<\rho^k_{{\mathcal N}_\tau}$.

Recall also that premice may be either passive or active, and active premice
are of three types, depending on the set of generators of the top
extenders. When we talk about embeddings between premice, we always require that
the premice in question are of the same kind according to this
categorization. We will say the premice are of the same type also in the case
where both are passive.

Now let $\bar{\tau}\in B_\tau$ if and only if the following hold:
\begin{enumerate}
\item
${\mathcal N}_{\bar{\tau}}$ is a strategic premouse of the same type as
${\mathcal N}_\tau$,
\item
$n_{\bar{\tau}}=n_\tau$, and
\item
There is a $\Sigma^{(n)}_0$-preserving
embedding $\sigma:{\mathcal N}_{\bar{\tau}}\to {\mathcal N}_\tau$ such that:
\begin{itemize}
\item
$\sigma\rst\bar{\tau}=\idm$,
\item
$\sigma(\bar{\tau})=\tau$ if $\bar{\tau}\in{\mathcal N}_{\bar{\tau}}$,
\item
$\sigma(p_{\bar{\tau}})=p_\tau$ and,
\item
For each $\xi\in p_\tau$, there is some generalized solidity witness $Q$ for
$\xi$ with respect to ${\mathcal N}_\tau$ and $p_\tau$, such that
$Q\in\rng(\sigma)$.
\end{itemize}
\item
$\alpha_{\bar{\tau}}=\alpha_\tau$.
\end{enumerate}

The following facts are proved by means of abstract fine structure theory, and
the proofs look exactly as the corresponding proofs in Zeman \cite{gls}:
\begin{itemize}
\item[(A)] $\sigma$ is unique, and we denote it by $\sigma_{\bar{\tau},\tau}$.
\item[(B)] $\sigma_{\bar{\tau},\tau}[\omega\rho^{n_{\bar{\tau}}}_{{\mathcal N}_{\bar{\tau}}}]$ is
                bounded in $\omega\rho^{n_\tau}_{N_\tau}$.
\item[(C)] If $\tau^*<\bar{\tau}$ are in $B_\tau$ then there is a unique map
         $\sigma_{\tau^*,\bar{\tau}}:{\mathcal N}_{\tau^*}\to{\mathcal
	 N}_{\bar{\tau}}$ with
	 the list of properties (1) -- (4) stipulated above, and with
	 $(\tau^*,\bar{\tau})$ in place of $(\bar{\tau},\tau)$.
\item[(D)] If $\tau^*<\bar{\tau}<\tau'$ are in $B_\tau\cup\{\tau\}$, then
       $\sigma_{\bar{\tau},\tau'}\circ\sigma_{\tau^*,\bar{\tau}}=
         \sigma_{\tau^*,\tau'}$.
\item[(E)] If $\bar{\tau}\in B_\tau$, then $B_\tau\cap\bar{\tau}=
                    B_{\bar{\tau}}-\min(B_\tau)$.
\end{itemize}

The proof of unboundedness of $B_\tau$, for $\tau$ of uncountable cofinality, is
similar to that in Zeman \cite{gls}, but uses the Condensation Lemma
\ref{n:l:condensation} where the argument from Zeman \cite{gls} used
condensation for $\Le$-models:

For $\lambda$ regular and sufficiently large, given $\tau'<\tau$, and working in
$\hod$, we construct a countable elementary substructure $X\prec H_\lambda$,
such that
 $$ {\mathcal N}_\tau,\tau',C^*\in X. $$
Letting $\tilde{\tau}=\sup(X\cap\tau)$, notice that $\tau'<\tilde{\tau}<\tau$ and $\tilde{\tau}\in C^*$, as
$C^*$ is closed. Let $\bar{{\mathcal N}}$ be the transitive collapse of ${\mathcal N}_\tau$, and let
$\sigma:\bar{{\mathcal N}}\to{\mathcal N}_\tau$ be the inverse of the collapsing isomorphism.
Set $\tilde{{\mathcal N}}=\ult^{n_\tau}(\bar{{\mathcal N}},\sigma\rst(\bar{{\mathcal N}}\res\bar{\tau}))$,
where $\sigma(\bar{\tau})=\tau$, and let $\sigma':\tilde{{\mathcal N}}\to{\mathcal N}_\tau$ be the
factor map, that is, $\sigma'\circ\tilde{\sigma}=\sigma$, and
$\sigma'\rst\tilde{\tau}=\idm$.

Exactly as in Zeman \cite{gls}, one can show the
following facts by means of abstract fine structural considerations:
\begin{itemize}
\item[(F)] $\tilde{\sigma}$ is $\Sigma^{(n)}_0$-preserving, and maps
               $\omega\rho^{n_\tau}_{\bar{{\mathcal N}}}$ cofinally into
	       $\omega\rho^{n_\tau}_{\tilde{{\mathcal N}}}$. Similarly,
	       $\sigma'$ is
	       $\Sigma^{(n)}_0$-preserving, but maps
               $\omega\rho^{n_\tau}_{\tilde{\mathcal N}}$ boundedly into
	       $\omega\rho^{n_\tau}_{{\mathcal N}}$.
\item[(G)] $\tilde{{\mathcal N}}$ is a sound and solid strategic premouse.
\item[(H)] $\tilde{{\mathcal N}}$ is a singularizing structure for $\tilde{\tau}$ with
               singularization degree $n_\tau$.
\item[(I)] $\sigma'(\tilde{\tau})=\tau$, $\sigma'(p_{\tilde{{\mathcal N}}})=p_\tau$, and
              $\alpha_\tau$ is the largest $\alpha<\tilde{\tau}$ satisfying
	      \[
              \tilde{h}^{n_\tau+1}_{\tilde{{\mathcal N}}}(\alpha\cup\{p_{\tilde{{\mathcal N}}}\})
                \cap\tilde{\tau}=\alpha.
              \]
\end{itemize}

Since the entire construction took place inside $\hod$,
Lemma~\ref{n:l:condensation} can be applied to the map
$\sigma':\tilde{{\mathcal N}}\to{\mathcal N}_\tau$. The rest follows again by abstract fine
structural considerations, literally as in Zeman \cite{gls}. In particular,
these considerations can be used to show that
$\tilde{{\mathcal N}}={\mathcal N}_{\tilde{\tau}}$ and $\sigma'(p_{\tilde{{\mathcal N}}})=p_\tau$, and
hence $\sigma'=\sigma_{\tilde{\tau},\tau}$. Additionally,
$\alpha_{\tilde{\tau}}=\alpha_\tau$. This shows that
$\tilde{\tau}\in B_\tau$.

In Zeman \cite{gls} it is only proved that $B_\tau$ is closed on a
tail-end; this was again caused by the fact that one has to consider
protomice. Here, we prove that $B_\tau$ is itself closed. 

Given a limit point $\tilde{\tau}$ of $B_\tau$, notice first that $\tilde{\tau}\in C^*$, let
$\tilde{{\mathcal N}}$ be the direct limit of
$\langle{\mathcal N}_{\bar{\tau}},\sigma_{\tau^*,\bar{\tau}}\mid
  \tau^*<\bar{\tau}<\tilde{\tau}\rangle$,
and let $\sigma':\tilde{{\mathcal N}}\to{\mathcal N}_\tau$ be the direct limit map. Again, by
abstract fine structural considerations that are essentially identical to those in Zeman
\cite{gls}, we establish the conclusions analogous to (F)--(I) for the
current version of $\tilde{{\mathcal N}}$ and $\sigma'$, and then as above apply
Lemma~\ref{n:l:condensation} to $\sigma':\tilde{{\mathcal N}}\to{\mathcal N}_\tau$, and
conclude that $\tilde{{\mathcal N}}$ is an initial segment of $\hod$. As above, we then
conclude that $\tilde{{\mathcal N}}={\mathcal N}_{\tilde{\tau}}$ and
$\sigma'=\sigma_{\tilde{\tau},\tau}$.

Having established closure and unboundedness of $B_\tau$, we follow the
construction from Zeman \cite{gls}, and obtain fully coherent sets $B^*_\tau$
by ``stacking'' the sets $B_\tau$; so
 $$ B^*_\tau=B_{\nu_0}\cup B_{\nu_1}\cup\dots\cup B_{\nu_{\ell_\tau}}, $$
where $\nu_0=\tau$, $\nu_{i+1}=\min(B_{\nu_i})$, and $\ell_\tau$ is the
least $\ell$ such that $B_{\nu_{\ell+1}}=\varnothing$. We then let $C^*_\tau$
be the set of all ordinals $\tau_\iota$, defined inductively as follows:
\begin{eqnarray*}
\tau_0&=&\min(B_\tau),\\
\xi^\tau_\iota&=&\mbox{the least $\xi<\tau$ such that
$\tilde{h}_\tau(\{\xi\}\cup\{p_\tau\})
\not\subseteq\rng(\sigma_{\tau_\iota,\tau})$},\\
\tau_{\iota+1}&=&\mbox{the least $\bar{\tau}<\tau$ such that
$\tilde{h}_\tau(\{\xi^\tau_\iota\}\cup\{p_\tau\})
\subseteq\rng(\sigma_{\bar{\tau},\tau})$},\\
\tau_\iota&=&\sup\{\tau_{\bar{\iota}}\mid\bar{\iota}<\iota\}\mbox{ for limit }\iota.
\end{eqnarray*}

The proof that the $C^*_\tau$ are fully coherent, unbounded in $\tau$ whenever
$\tau$ has uncountable cofinality in $\hod$, and $\otp(C^*_\tau)<\tau$, can be
carried out using abstract fine structural considerations, and is essentially the same as
the argument in Zeman \cite{gls}.

Finally, we let $C_\tau=\lim(C^*_\tau)$ whenever
$\lim(C^*_\tau)$ is unbounded in $\tau$, and otherwise let $C_\tau$ be some randomly
chosen cofinal $\omega$-sequence in $\tau$. (Recall again that here
our notation diverges from that in Zeman \cite{gls}, where $C_\tau$ denoted sets that
are fully coherent but not necessarily cofinal at countably cofinal $\tau$.)
\end{proof}

\begin{remark}
Above, we worked in the theory $\ADR$, as it directly relates to our negative results.
However, as long as we are in the situation where there is no $\Gamma\subseteq{\mathcal P}(\R)$ such that
$\mathrm{L}(\Gamma,\R)\models\ADR + \mbox{\rm``$\Theta$ is Mahlo in $\hod$''}$, Theorem \ref{hyplowerbound} holds
assuming only $\AD^+$.

We briefly sketch the argument in the case where $\Theta=\theta_{\Omega+1}$
for some ordinal $\Omega$. In this case we use the following facts. 
\begin{enumerate}
\item[(a)] Models that admit the $\hod$ analysis are of the form
           $\Lp^\Sigma({\mathbb R})$. 
\item[(b)] Under our minimality assumption, the $\hod$ analysis applies to our
           universe $\V$. 
\end{enumerate}
That (a) and (b) hold follows from \cite{SarSteel}. 
Assuming them, let $G$ be $\pmax$-generic over $\mbfV$. By the $S$-construction from Section 2.11 of Sargsyan \cite{S}, the
model $\Lp^\Sigma({\mathbb R})[G]$ can
be rearranged into the form $\Lp^\Sigma({\mathbb R},G)$. Since $G$ well-orders
${\mathbb R}$ in order type $\omega_2$, there is $A\subseteq\omega_2$ such that
$\Lp^\Sigma({\mathbb R},G)=\Lp(A)$. The standard construction of the canonical
$\square_{\omega_2}$-sequence in $\Lp(A)$ thus yields a
$\square_{\omega_2}$-sequence in the $\pmax$-extension.
\end{remark}

Suppose now that we have a sequence 
$\langle c_\tau\mid\tau\in C^*\rangle$ as in Theorem~\ref{hyplowerbound}, and 
that the Axiom of Choice holds. We may assume in addition that $c_\tau$ consists of
successor ordinals whenever $c_\tau\not\subseteq C^*$. It is a well-known fact 
that if $\alpha<\beta$ are ordinals
in the interval $(\omega_2,\omega_3)$ then there is a coherent sequence 
$\langle d_\tau\mid\tau\in\lim\cap(\alpha,\beta)\rangle$ such that each
$d_\tau$ is a closed unbounded subset of $\tau$ contained in the interval
$(\alpha,\beta)$ and has order-type at most $\omega_2$. (This is proved by
induction on the length of initial segments of such a sequence.) Using Choice 
we can pick such a sequence $\langle d_\tau\rangle_\tau$ for
each interval $(\alpha,\beta)$ where $\alpha<\beta$ are two adjacent elements
of $C^*$, and assemble them together with the sequence 
$\langle c_\tau\mid\tau\in C^*\rangle$ into a single coherent sequence
$\langle c'_\tau\mid\tau\in\lim\cap(\omega_2,\omega_3)\rangle$. This  sequence
is almost a $\square_{\omega_2}$-sequence, except that it satisfies a bit 
more relaxed restriction on order-types, namely $\otp(c'_\tau)<\tau$ for every
limit ordinal $\tau\in(\omega_2,\omega_3)$. For auxiliary purposes, set also
$c'_\tau=\tau$ whenever $\tau\le\omega_2$ is a limit ordinal. It is then easy
to verify that the following recursive construction from Jensen \cite{J} turns
the sequence $\langle c'_\tau\mid\tau\in\lim\cap(\omega_2,\omega_3)\rangle$
into a  $\square_{\omega_2}$-sequence
$\langle c^*_\tau\mid\tau\in\lim\cap(\omega_2,\omega_3)\rangle$. Let  
$\pi_\tau:\otp(c'_\tau)\to c'_\tau$ be the unique order isomorphism. Then set   
\[
c^*_\tau=\pi_\tau[c^*_{\otp(c'_\tau)}]. 
\]
As the partial order $\pmax * \Add(\Theta, 1)$ forces Choice, this argument gives the following corollary of Theorem 
\ref{hyplowerbound}. 

\begin{corollary}\label{hyplowerboundcor}
Assume that 
$\ADR$ holds
and that there is no $\Gamma \subseteq \mathcal{P}(\mathbb{R})$ such that
$\mathrm{L}(\Gamma, \mathbb{R}) \models \text{\rm``$\Theta$ is Mahlo in
$\HOD$"}$.  Then $\square_{\omega_{2}}$ holds after forcing with 
$\pmax * \Add(\Theta, 1)$. 
\end{corollary}

\section{Choiceless extensions where square fails} \label{section:nochoice}

In this section we present two results showing that $\omega_3$ is square inaccessible, and even
threadable (see Definition
\ref{def:threadable}), in the $\pmax$ extension of suitable models of determinacy. These arguments do not show that the subsequent forcing
$\Add(\omega_{3},1)$, which adds a Cohen subset of $\omega_3$, and in the process well-orders
$\mathcal{P}(\mathbb{R})$, does not (``accidentally'') add a square sequence. In fact it can, as the hypotheses 
of Theorems \ref{hyplowerbound} and \ref{theorem:gone} are jointly consistent relative to suitable large cardinals. Moreover, the hypothesis of 
Theorem \ref{theorem:gone} (with $\theta$ as $\Theta$) follows from
$\AD^{+}$ + $\Unif$ + $\V\mathord{=}\mathrm{L}(\mathcal{P}(\R)) + ``\Theta \text{ is Mahlo in }\HOD$'' (and is strictly weaker than it, if it is consistent).
To see this, note that the Solovay sequence is closed, and, by Remark \ref{nowoname}, $\AD^{+}$ + $\Unif$ implies that the Solovay sequence has limit length. 
Finally, an application of Theorem \ref{voplem} gives a model of the form $L_{\kappa}(\mathcal{P}_{\theta_{\alpha}}(\mathbb{R}))$ satisfying the hypothesis of 
Theorem \ref{theorem:gone}.

Given a
set $S$ of ordinals, let us say that a sequence of sets 
$\langle x_\xi\mid\xi\in S\rangle$ is a \emph{singularizing sequence} for $S$ if
each $x_\xi$ is a cofinal subset of $\xi$ of order-type strictly smaller than
$\xi$. 

\begin{theorem} \label{theorem:gone}
Assume that $\AD^{+}$ holds, and that $\theta$ is a limit element of the Solovay sequence. Let $R_{\theta}$ be the set of  $\kappa<\theta$ which are 
both limit elements of the Solovay sequence and regular in $\HOD$. 
Assume that $R_{\theta}$ is unbounded in $\theta$, and let $R\subseteq R_\theta$ be an unbounded subset of
$\theta$ in $\HOD$. Then there does not exist a singularizing sequence for $R$ in the
$\pmax$ extension of $\HOD_{\mathcal{P}_{\theta}(\mathbb{R})}$.  

In particular, $\omega_3$ is square inaccessible in the $\pmax$ extension of
$\HOD_{\mathcal{P}_{\theta}(\mathbb{R})}$.
\end{theorem}

\begin{proof}
In $\HOD_{\mathcal{P}_{\theta}(\mathbb{R})}$ fix a $\pmax$-name $\sigma$ for a
singularizing sequence for $R$. Fix $\kappa\in R$ such that $\sigma$ is ordinal
definable from some $A\in{\mathcal P}_\kappa(\R)$. By Theorem \ref{regpres}, 
$\kappa$ is regular in $\HOD_{\mathcal{P}_{\kappa}(\mathbb{R})}$. The canonical $\pmax$-name  
\[
\dot{x}_\kappa=\{(p,\check{\alpha})\mid
p\Vdash\alpha\in\sigma_{\check{\kappa}}\} 
\]
for the member of the interpretation of $\sigma$ associated with $\kappa$
is definable from $\sigma$ and $\kappa$, and therefore belongs to
$\HOD_{\mathcal{P}_{\kappa}(\mathbb{R})}$. 

Then $x_\kappa$, the interpretation of $\dot{x}_\kappa$, belongs to the
$\pmax$ extension of 
$\HOD_{\mathcal{P}_{\kappa}(\mathbb{R})}$. But this is a contradiction, since,
on the one hand, $x_\kappa$ is unbounded in $\kappa$ and has order type less
than $\kappa$, and, on the other hand, the regularity of $\kappa$ is preserved
in the $\pmax$ extension of $\HOD_{\mathcal{P}_{\kappa}(\mathbb{R})}$, as
$|\pmax|={\mathfrak c}$. 

The square-inaccessibility of $\omega_3$ in the $\pmax$ extension of
$\HOD_{\mathcal{P}_{\theta}(\mathbb{R})}$ follows from the fact that any
square sequence witnessing $\square_{\omega_3}$ induces a singularizing
sequence for $R_\theta$. 
\end{proof}

\begin{remark}\label{add-square}
Let us emphasize that the negative result in Theorem~\ref{theorem:gone} is a
negative result on the existence of a singularizing sequence rather than a
negative result on square sequences. The contradiction in the proof of the theorem is
obtained not from a combinatorial situation that would block the existence of
a square sequence, rather this contradiction exhibits the lack of choice in
the $\pmax$ extension: Every $\kappa\in R_\theta$ is singular, but we cannot
pick a singularizing sequence.  The coherence requirement does not play any
role here. 
\end{remark} 

An uncountable cardinal $\theta$ is \emph{weakly compact} if 
for every $S \subset \mathcal{P}(\theta)$ of cardinality $\theta$ there is a
$\theta$-complete filter $F$ on $\theta$ such that $\{ A, \theta \setminus A\} \cap F \neq \emptyset$ for
each $A \in S$ (this is not the usual definition, but it easily seen to be equivalent to the Extension Property in 
Theorem 4.5 of \cite{kanamori}). It follows easily from this formulation (or, more directly, the Extension Property)
that $\theta$ is threadable if it is weakly compact.

\begin{theorem} \label{theorem:weaklycompact}
Assume that $\AD^{+}$ holds, and that $\theta$ is a limit element of the Solovay sequence, and weakly compact
in $\HOD$. Then $\theta$ remains weakly compact in the $\pmax$ extension of
$\HOD_{\mathcal{P}_{\theta}(\mathbb{R})}$.
\end{theorem}

\begin{proof}
Let $\tau$ be a $\pmax$ name in $\HOD_{\mathcal{P}_{\theta}(\mathbb{R})}$ for a collection of $\theta$
many subsets of $\theta$. Then $\tau$ is ordinal definable in $\HOD_{\mathcal{P}_{\theta}(\mathbb{R})}$ from a 
set of reals that is itself definable from a bounded subset $S$ of $\theta$. For each $\pmax$ condition 
$p$ and each ordinal $\alpha < \theta$, let $A_{p,\alpha}$ be the set of ordinals forced by $p$ to be in 
the $\alpha$-th set represented by $\tau$.

Fix an ordinal $\delta < \theta$ such that $S \subseteq \delta$, and let
 $$ C \colon \mathcal{P}(\omega) \times \mathcal{P}(\delta) \to \mathcal{P}(\omega + \delta) $$
be defined by letting $C(a,b) \cap \omega = a$, and, for all $\gamma < \delta$, $\omega + \gamma \in
C(a,b)$ if and only if $\gamma \in b$. Fix also a coding of elements of $H(\aleph_{1})$ by subsets of
$\omega$.

By Lemma \ref{voplem}, there is a partial order $Q$ of cardinality less than $\theta$ in $\HOD$
such that for each $a\subseteq \omega$, $C(a,S)$ is $\HOD$-generic for $Q$. The proof of Lemma 
\ref{voplem} shows that there is a sequence $\langle \sigma_{\alpha} \mid \alpha < \theta \rangle$ in 
$\HOD$, consisting of $Q$-names, such that for each $\alpha < \theta$ and each $a \subseteq \omega$ 
coding some $p \in \pmax$, the realization of $\sigma_{\alpha}$ by $C(a,S)$ is the set $A_{p,\alpha}$.
To see this, following the argument from Lemma \ref{voplem}, let $T$ be an ordinal definable relation on
$\mathcal{P}(\omega + \delta) \times \theta \times \theta$ such that for all $a \subseteq \omega$ coding
$p \in \pmax$ and all $(\alpha, \beta) \in \theta \times \theta$, we have that $\beta \in A_{p,\alpha}$ if and
only if $T(C(a,S), \alpha, \beta)$. Define $T^{*}$ on $\theta \times\theta \times Q$ by letting
$T^{*}(\alpha, \beta, q)$ hold if and only if $T(D, \alpha, \beta)$ holds for all $D \in h(q)$, where $h$ is the
function from the proof of Lemma \ref{voplem}. Then $T^{*}$ is (essentially) the desired sequence 
$\langle \sigma_{\alpha} : \alpha < \theta\rangle$. 

Applying the weak compactness of $\theta$ in $\HOD$, we can find in $\HOD$ a $\theta$-complete filter $F$ on
$\theta$, such that, for each $\alpha < \theta$ and each $q \in Q$, $F$ contains either the set
$\{ \beta < \theta \mid q \forces \check{\beta} \in \sigma_{\alpha}\}$, or its complement. Since $|Q| < \theta$
in $\HOD$, for each $p \in \pmax$ and each $\alpha < \theta$, either $A_{p, \alpha}$ or its complement contains a set
in $F$. By Theorem \ref{getmeas}, the filter generated by $F$ is $\mathbb{R}$-complete in
$\HOD_{\mathcal{P}_{\theta}(\mathbb{R})}$. It follows then that the filter generated by $F$ measures all the sets in the
realization of $\tau$ in the $\pmax$ extension.
\end{proof}

\section{Forcing the square inaccessibility of $\omega_3$} \label{section:omega2}

A \emph{partial} $\square_{\kappa}$-\emph{sequence}
is a sequence $\langle C_{\alpha} \mid \alpha \in A \rangle$ (for $A$ a subset of $\kappa^{+}$)
satisfying the three conditions in Definition \ref{squaresubkappa} for each $\alpha \in A$. Note that
condition (\ref{subcoherence}) implies that $\beta \in A$ whenever $\alpha \in A$ and $\beta$ is a limit
point of $C_{\alpha}$.

The partial order $\Add(\omega_{3},1)$ adds a subset of $\omega_{3}$ by initial segments.
When $\mathfrak{c} = \aleph_{2}$, as in a $\pmax$ extension, $\Add(\omega_{3},1)$ well-orders
$\mathcal{P}(\mathbb{R})$ in order type $\omega_{3}$. 

The hypotheses of Theorem \ref{theorem:sqcofoo} below imply
that $\Theta$ is regular, since any singularizing function would exist in $\HOD_{A}$ for some set of reals
$A$, and, by Theorem \ref{regpres}, this would give a club of singular cardinals in $\HOD$ below
$\Theta$.
Modulo the $\HOD$ analysis (and the assumption of $\ADR$ + $\mathrm{V} = \mathrm{L}(\mathcal{P}(\mathbb{R}))$),
the hypothesis of Theorem \ref{theorem:sqcofoo} is close to the negation of the hypothesis of Theorem \ref{hyplowerbound} (see Remark \ref{remark:sqcofoo}).

\begin{theorem} \label{theorem:sqcofoo}
Assume that $\AD^{+}$ + $\Unif$ holds, that $\V = \mathrm{L}(\mathcal{P}(\mathbb{R}))$, and that
stationarily many elements $\theta$ of cofinality $\omega_{1}$ in the Solovay sequence
are regular in $\HOD$.
Then in the $\pmax * \Add(\omega_{3},1)$-extension there is no
partial $\square_{\omega_2}$-sequence defined
on all points of cofinality at most $\omega_1$.
\end{theorem}

\begin{proof}
Let $\kappa$ be a regular cardinal which is not the surjective image of $\mathcal{P}(\mathbb{R})$, and let 
$T$ be the theory of $\mathrm{L}_{\kappa}(\mathcal{P}(\mathbb{R}))$. 
Suppose that $\tau$ is a $\pmax * \Add(\omega_{3},1)$-name in $\mathrm{L}(\mathcal{P}(\mathbb{R}))$
whose realization is forced by some condition $p_{0}$
to be such a partial $\square_{\omega_2}$-sequence. We may assume that $\tau$ is coded by a subset
of $\mathcal{P}(\mathbb{R})$, and (by using the least ordinal parameter
defining a counterexample to the theorem) that $(\tau, p_{0})$ is definable in $\mathrm{L}_{\kappa}(\mathcal{P}(\mathbb{R}))$ from some $A \subseteq
\mathbb{R}$. Using our hypothesis (and Theorem \ref{regpres} for item (\ref{regitem})), we
get $\theta < \Theta$ with $A$ in $\mathcal{P}_{\theta}(\mathbb{R})$, and ordinals $\xi_{0}$ and
$\xi_{1}$ such that
\begin{enumerate}
\item
$\theta < \xi_{0} < \Theta < \xi_{1} \leq \Theta^{+}$;
\item
$\mathrm{L}_{\xi_{1}}(\mathcal{P}(\mathbb{R}))$ satisfies $T$;
\item \label{goodcof}
$\theta$ is a limit element of the Solovay sequence of $\mathrm{L}(\mathcal{P}(\mathbb{R}))$ of cofinality
$\omega_{1}$;
\item \label{regitem}
$\theta$ is regular in $\HOD_{\mathcal{P}_{\theta}(\mathbb{R})}$;
\item
$p_{0} \in \mathrm{L}_{\xi_{0}}(\mathcal{P}_{\theta}(\mathbb{R}))$;
\item\label{definablecondition} \label{evel}
every element of $\mathrm{L}_{\xi_{0}}(\mathcal{P}_{\theta}(\mathbb{R}))$ is definable in
$\mathrm{L}_{\xi_{0}}(\mathcal{P}_{\theta}(\mathbb{R}))$ from a set of reals in
$\mathcal{P}_{\theta}(\mathbb{R})$;
\item
in $\mathrm{L}_{\xi_{1}}(\mathcal{P}(\mathbb{R}))$, $\tau$ is a $\pmax * \Add(\omega_{3},1)$-name
whose realization is forced by $p_{0}$ to be a partial $\square_{\omega_2}$-sequence defined  on the
ordinals of cofinality at most $\omega_1$; and
\item
there exist $\sigma \in \mathrm{L}_{\xi_{0}}(\mathcal{P}_{\theta}(\mathbb{R}))$ and an elementary
embedding
 $$ j\colon \mathrm{L}_{\xi_{0}}(\mathcal{P}_{\theta}(\mathbb{R})) \to \mathrm{L}_{\xi_{1}}(\mathcal{P}(\mathbb{R})) $$
with critical point $\theta$ such that $j(\sigma) = \tau$.
\end{enumerate}

To see this, let $\xi_{1}$ be the least ordinal $\xi > \Theta$ such that $\mathrm{L}_{\xi}(\mathcal{P}(\mathbb{R}))
\models T$, and $\tau$ is definable from $A$ in $\mathrm{L}_{\xi}(\mathcal{P}(\mathbb{R}))$. Then every element
of $\mathrm{L}_{\xi_{1}}(\mathcal{P}(\mathbb{R}))$ is definable in $\mathrm{L}_{\xi_{1}}(\mathcal{P}(\mathbb{R}))$ from a
set of reals. For each $\alpha < \Theta$, let $X_{\alpha}$ be the set of elements of $\mathrm{L}_{\xi_{1}}
(\mathcal{P}(\mathbb{R}))$ definable from a set of reals of Wadge rank less than $\alpha$.
Then, by the definition of the Solovay sequence, and Remark \ref{Solovayremark}, the order type of
$X_{\alpha}\cap\Theta$ is always less than the least element of the Solovay sequence above $\alpha$.
Since $\Theta$ is regular, $X_{\alpha}\cap\Theta$ is bounded below $\Theta$. Let $f(\alpha) =
\sup(X_{\alpha} \cap \Theta)$. Then $f$ is continuous, and we can find $\theta$ satisfying items
(\ref{goodcof}) and (\ref{regitem}) above, and such that $f(\theta) = \theta$. Let $\xi_{0}$ be the
order type of $X_{\theta} \cap \xi_{1}$, so that $\mathrm{L}_{\xi_{0}}(\mathcal{P}_{\theta}(\mathbb{R}))$ is the
transitive collapse of $X_{\theta}$, the embedding $j$ is simply the inverse of the collapse, and $\sigma$ is the collapse of $\tau$.

Let $M_{0} = \mathrm{L}_{\xi_{0}}(\mathcal{P}_{\theta}(\mathbb{R}))$ and $M_{1} = \mathrm{L}_{\xi_{1}}(\mathcal{P}
(\mathbb{R}))$. Let $G$ be $\pmax$-generic over $M_{1}$, containing the first coordinate of $p_{0}$.
Then $j$ lifts to
 $$ j\colon M_0[G]\to M_{1}[G], $$
and $j(\Add(\theta,1)^{M_0[G]}) = \Add(\Theta,1)^{M_{1}[G]}$. Since $\pmax$ is countably closed,
$\theta$ has cofinality $\omega_{1}$ in $M_{1}[G]$. It follows that each countable subset of
$\Add(\theta,1)^{M_0[G]}$ in $M_{1}[G]$ is an element of $\mathrm{L}(A,\mathbb{R})[G]$ for some
$A \in \mathcal{P}(\mathbb{R})^{M_{0}}$, and therefore that
$\Add(\theta,1)^{M_0[G]}$ is $\omega$-closed in $M_{1}[G]$.
Let $z$ be the second coordinate of $p_{0}$, as realized by $G$. 

\begin{claim}\label{claim6.2}
If $H$ is $M_{1}[G]$-generic over $\Add(\theta,1)^{M_0[G]}$, with $z \in H$, then
$\sigma_{G*H}$ has a thread in $M_{1}[G][H]$.
\end{claim}

We first finish the proof of the theorem, assuming the existence of an $H$ as in the claim. Since $\theta$ has uncountable cofinality in $M_{1}[G][H]$, there can be at most one thread through
$\sigma_{G*H}$ in $M_{1}[G][H]$. The thread, being unique, would be in
$\HOD_{\mathcal{P}_{\theta}(\mathbb{R})}[G][H]$ (note that $\HOD_{\mathcal{P}_{\theta}(\mathbb{R})}$
has the same sets of reals as $M_{0}$).
This leads to a contradiction, as
$\theta$ would be collapsed in $\HOD_{\mathcal{P}_{\theta}(\mathbb{R})}[G][H]$, which is impossible
since $\theta$ is regular in $\HOD_{\mathcal{P}_{\theta}(\mathbb{R})}$.

It suffices then to prove the claim.

\begin{proof}[Proof of Claim \ref{claim6.2}.]
Toward a contradiction, suppose that the claim were false. 
Let $C\in M_1$ be club in $\theta$, with ordertype $\omega_1$. By the Coding
Lemma, $C$ is in $\mathrm{L}(B,\mathbb{R})$ for some set of reals $B$ with $|B|_W=\theta$. If $H$ is
$M_{1}[G]$-generic as above, with $z \in H$, then $\sigma_{G*H}$ is a coherent sequence of length $\theta$ with no
thread in $M_{1}[G][H]$. We can fix $(\pmax * \Add(\theta,1)^{M_0[G]})$-names $\rho$ and $\psi$ in
$\mathrm{L}(B, \mathbb{R})$ such that
\begin{itemize}
\item
$\rho_{G*H}$ is the tree of attempts to build a thread through $\sigma_{G*H}$ along $C$ (i.e., the
relation consisting of those pairs $(\alpha, \beta)$ from $C$ for which the $\beta$-th member of
$\sigma_{G*H}$ extends the $\alpha$-th member), and
\item
$\psi_{G*H}$ is the poset that specializes $\rho_{G*H}$ (i.e., which consists of finite partial functions
mapping $C$ to $\omega$ in such a way that $\rho_{G*H}$-compatible elements of $C$ are mapped to
distinct elements of $\omega$, ordered by inclusion).
\end{itemize}
Since $\sigma_{G*H}$ is forced by $p_{0}$ to have no thread in $M_{1}[G][H]$, $\Add(\theta,1)^{M_0[G]} * \psi_{G*H}$ is
$\omega$-closed$*$c.c.c., and thus proper, in $M_{1}[G]$.

\begin{subclaim}
In $M_1[G]$, there are $H,f$ such that
\begin{itemize}
\item $H$ is $\Add(\theta,1)^{M_0[G]}$-generic over $M_0[G]$, with $z \in H$, and
\item  $f$ specializes $\rho_{G*H}$.
\end{itemize}
\end{subclaim}

\begin{proof}
In $M_0[G]$, $\Add(\theta,1)$ is $< \theta$-closed.

For each $\alpha \in C$ and each ternary formula $\phi$, (recalling condition (\ref{definablecondition}) from the choice
of $\xi_{0}$ and $\theta$) let $E_{\alpha,\phi}$ be the collection of sets
of the form $\{ x \mid M_{0}[G] \models \phi(A, G, x)\}$ which are dense subsets of
$\Add(\theta,1)^{M_{0}[G]}$, where $A$ is an element of $\mathcal{P}_{\alpha}(\mathbb{R})$.
As $\mathcal{P}_{\alpha}(\mathbb{R})$ is a surjective image of $\mathbb{R}$ in $M_{0}$, $E_{\alpha,\phi}$ has cardinality less than $\theta$ in $M_{0}[G]$.
It follows that there is a dense subset
$D_{\alpha,\phi}$ of $\Add(\theta,1)^{M_{0}[G]}$ refining all the members of $E_{\alpha,\phi}$.

Since $\PFA(\mathfrak{c})$ holds in $M_1[G]$ (by Woodin \cite[Theorem 9.39]{W}), there is a filter
$H * K$ on $\Add(\theta,1)^{M_0[G]} * \psi$ in $M_{1}[G]$ such that $H$ meets each $D_{\alpha,\phi}$, and
$H*K$ meets the dense sets guaranteeing that $K$ determines a specializing function $f$ for
$\rho_{G*H}$. This gives the subclaim.
\end{proof}

Let $H$ and $f$ be as in the subclaim. Then in $M_1[G]$, $H$ is a condition in $\Add(\Theta,1)$. We
can therefore find a generic $H_1$ over $M_1[G]$ such that $j$ lifts to
 $$ j\colon M_0[G][H]\to M_1[G][H_1]. $$

But then $\sigma_{G*H}$ has the thread $(\tau_{G*H_1})_\theta$. But $f$ is in $M_1[G]$, so
$\omega_1$ was collapsed by going to $M_1[G][H_1]$, giving a contradiction.
\end{proof}

This completes the proof of Theorem \ref{theorem:sqcofoo}.
\end{proof}

\begin{definition}\label{squaresubkappalambda}
Given cardinals $\kappa$ and $\lambda$, the principle $\square_{\kappa, \lambda}$ asserts the existence of a sequence $\langle \mathcal{C}_{\alpha} \mid \alpha < \kappa^{+} 
\rangle$ of nonempty sets such that for each $\alpha < \kappa^{+}$,
\begin{enumerate}
\item
$|\mathcal{C}_{\alpha}| \leq \lambda$;
\item
each element of $\mathcal{C}_{\alpha}$ is club in $\alpha$, and has order type at
most $\kappa$;
\item\label{closurecondition}
for each member $C$ of $\mathcal{C}_{\alpha}$, and each limit point $\beta$ of $C$,
$C \cap \beta \in\mathcal{C}_{\beta}$.
\end{enumerate}
\end{definition}

We call a sequence witnessing the above principle a $\square_{\kappa,\lambda}$-sequence.
The statement $\square_{\omega_{2}, \omega_{2}}$ follows from $2^{\aleph_{1}} = \aleph_{2}$ \cite{J}, which holds in 
$\pmax$ extensions. 

\begin{question}
Can one improve Theorems \ref{theorem:sqcofoo} to obtain the failure of
$\square_{\omega_{2}, \omega}$?
\end{question}

\begin{remark}\label{remark:sqcofoo}
In terms of consistency strength, the hypothesis of Theorem \ref{theorem:sqcofoo} is below a Woodin limit of
Woodin cardinals; this follows from the proof of Sargsyan \cite[Theorem, 3.7.3]{S} (see also \cite{S15a, S15b}). Furthermore, the
hypothesis is equiconsistent with a determinacy statement that is easier to state. In fact, something
stronger than mere equiconsistency holds, as we proceed to sketch.
\end{remark}

\begin{theorem} Assume $\AD^{+} + \Unif$ +$\V=\mathrm{L}({\mathcal P}(\R))$, and that the $\HOD$ analysis applies. Then the following are equivalent.
\begin{enumerate}
\item $\HOD\models ``\Theta$ is Mahlo to measurables".
\item $\Theta$ is Mahlo to measurables of $\HOD$.
\item There are stationarily many $\theta$ in the Solovay sequence that have cofinality $\omega_1$ and are regular in
$\HOD$.
\end{enumerate}
\end{theorem}
\begin{proof}
That clause 2 implies clause 1 is immediate. Assume now that clause 1 holds. 
Let 
\[S=\{ \theta_\alpha<\Theta: \HOD\models ``\theta_\alpha \text{ is measurable''}\}.
\] We have that $\HOD\models ``S$ is stationary". It follows from Theorem \ref{regpres} that $S$ is stationary in $\V$. Hence, clause 1 implies clause 2.

Next suppose $\Theta$ is Mahlo to measurables of $\HOD$. Let $S$ be as above. It follows from the $\HOD$ analysis (in particular, see \cite[Lemma 8.25]{Steel}) that every member 
of $S$ has cofinality $\omega_1$. This shows that clause 2 implies clause 3.

Finally, assume that there are stationarily many $\theta$ in the Solovay sequence that have cofinality $\omega_1$ and are regular in
$\HOD$. Let 
\[S=\{ \theta_\alpha<\Theta: \cf(\theta_\alpha)=\omega_1 \text{ and }\HOD\models ``\theta_\alpha\text{ is regular"}\}.
\] 
It follows from the $\HOD$ analysis that each $\theta_\alpha\in S$ is measurable in $\HOD$. To see this, fix $\theta_\alpha\in S$ and let $(\mathcal{Q}, \Lambda)\in \mathcal{F}$ be 
such that $\theta_\alpha\in \mathcal{M}_\infty(\mathcal{Q}, \Lambda)$ and such that for some $\kappa\in \mathcal{Q}$, $\pi^{\Lambda}_{\mathcal{Q}, \infty}(\kappa)=\theta_\alpha$. 
Since $\theta_\alpha$ is regular in $\HOD$ we have that $\kappa$ is regular in $\mathcal{Q}$. If $\kappa$ is not measurable in $\mathcal{Q}$ then $\pi^\Lambda_{\mathcal{Q}, \infty}
[\kappa]$ is cofinal in $\theta_\alpha$ implying that $\theta_\alpha$ has countable cofinality. Therefore, we must have that $\kappa$ is measurable in $\mathcal{Q}$ implying that 
$\HOD\models ``\theta_\alpha$ is measurable". This shows that clause 3 implies clause 1.
\end{proof}

\section{Stronger hypotheses and the threadability of $\omega_3$} \label{section:stronger}

In this section we apply a hypothesis stronger than the one used in Theorem
\ref{theorem:sqcofoo} to obtain the failure of a weakening of $\square(\omega_{3})$. 

\begin{definition}
Given an ordinal $\gamma$ and a cardinal $\lambda$, the principle $\square(\gamma, \lambda)$ asserts
the existence of a sequence $\langle \mathcal{C}_{\alpha} \mid \alpha < \gamma \rangle$ satisfying the following conditions. 
\begin{enumerate}
\item
For each $\alpha < \gamma$,
\begin{itemize}
\item
$0 < |\mathcal{C}_{\alpha}| \leq \lambda$;
\item
each element of $\mathcal{C}_{\alpha}$ is club in $\alpha$;
\item
for each member $C$ of $\mathcal{C}_{\alpha}$, and each limit point $\beta$ of $C$, $C \cap
\beta \in\mathcal{C}_{\beta}$.
\end{itemize}
\item
There is no thread through the sequence, that is, there is no club $E \subseteq
\gamma$ such that $E \cap \alpha \in \mathcal{C}_{\alpha}$ for every limit point $\alpha$ of $E$.
\end{enumerate}
\end{definition}

Again, we refer to sequences witnessing the above principle as $\square(\gamma,
\lambda)$-sequences. Notice that $\square_{\kappa, \lambda}$ implies $\square(\kappa^{+},
\lambda)$.
The arguments of Todorcevic \cite{T1, T2} mentioned above show that $\MM(\mathfrak{c})$ implies the failure of
$\square(\gamma, \omega_{1})$ for any ordinal $\gamma$ of cofinality at least $\omega_{2}$.

Before stating our result, we recall a result of Woodin that will be useful in what follows. The
notion of {\em $A$-iterability} for $A$ a set of reals, crucial in the theory of $\Pmax$, is introduced in
Woodin \cite[Definition 3.30]{W}. For $X\prec H(\omega_2)$, Woodin denotes by $M_X$ its transitive
collapse. Woodin \cite[\S 3.1]{W} presents a series of covering theorems for $\Pmax$ extensions. In
particular, we have:

\begin{theorem}[Woodin {\cite[Theorem 3.45]{W}}] \label{thm:woodin}
Suppose that $M$ is a proper class inner model that contains all the reals and satisfies $\AD+\DC$.
Suppose that for any $A\in{\mathcal P}(\R)\cap M$, the set
$$ \{X\prec H(\omega_2)\mid \mbox{\rm$X$ is countable, and $M_X$ is $A$-iterable}\} $$
is stationary. Let $X$ in $\V$ be a bounded subset of $\Theta^M$ of size $\omega_1$. Then there is a set
$Y\in M$, of size $\omega_1$ in $M$, and such that $X\subseteq Y$.
\end{theorem}

Typically, we apply this result as follows: We start with $M=\mathrm{L}({\mathcal P}(\R))$, a model of $\AD^++
\DC$, and force with $\Pmax$ to produce an extension $M[G]$. The technical stationarity assumption is
then true in $M[G]$ by virtue of Woodin \cite[Theorem 9.32]{W}, and we can then apply Theorem
\ref{thm:woodin} in this setting.

The following lemma will be used in the proof of Theorem \ref{theorem:vanilla}.
A similar fact was used in the proof of Theorem \ref{theorem:sqcofoo}. 

\begin{lemma} \label{lemma:closurelemma}
Suppose that $M_{0} \subseteq M_{1}$ are models of $\ZF + \AD^{+}$ such that
\begin{itemize}
\item $\mathbb{R} \subseteq M_{0}$;
\item $\mathcal{P}(\mathbb{R})^{M_{0}}$ is a proper subset of $\mathcal{P}(\mathbb{R})^{M_{1}}$;
\item
 $M_{1} \models \mbox{\rm``$\Theta$ is regular"}$;
\item
 $\Theta^{M_{0}}$ has cofinality at least $\omega_{2}$ in $M_{1}$.
\end{itemize}
Let $G \subset \pmax$ be an $M_{1}$-generic filter. Then $\Add(\omega_{3},1)^{M_{0}[G]}$ is closed
under $\omega_{1}$-sequences in $M_{1}[G]$.
\end{lemma}

\begin{proof}
Since $M_{0} \models \AD^{+}$ and $G$ is an $M_{0}$-generic filter for $\pmax$,  $\omega_{3}^{M_{0}
[G]} = \Theta^{M_{0}}$.
We show first that $\cof(\Theta^{M_{0}}) = \omega_{2}$ in $M_{1}[G]$. Since we have assumed that
$\cof(\Theta^{M_{0}}) \geq \omega_{2}$ in
$M_{1}$, equality is a consequence of the Covering Theorem \ref{thm:woodin}, as follows. Given any bounded subset $X$
of $\Theta^{M_{1}}$ of cardinality $\aleph_{1}$ in $M_{1}[G]$, let $A\in M_1$ be a set of reals of Wadge
rank at least $\sup(X)$, and apply Theorem \ref{thm:woodin} with $M$ as $\mathrm{L}(A,\R)$ and $\V$ as $L(A, \R)[G]$.
Since $\mathcal{P}(\omega_{1})^{M_{1}[G]}$ is contained in $\mathrm{L}(\mathbb{R})[G]$
(Woodin \cite[Theorem 9.32]{W}), $X$ will be an element of $\mathrm{L}(A, \mathbb{R})[G]$, and thus a subset of an element 
of $\mathrm{L}(A, \mathbb{R})$ of cardinality $\aleph_{1}$ in $\mathrm{L}(A, \mathbb{R})$. 

Let $\langle p_{\alpha}  \mid \alpha < \omega_{1} \rangle$ be a sequence in $M_{1}[G]$ consisting of
conditions in $\Add(\omega_{3},1)^{M_{0}[G]}$. Since $\cof(\Theta^{M_{0}}) = \omega_{2}$ in
$M_{1}[G]$, we may fix a $\gamma < \Theta^{M_{0}}$ such that each $p_{\alpha}$ is a subset of
$\gamma$. Since $\omega_{2}$-$\DC$ holds in $M_{1}[G]$ by Woodin \cite[Theorem 9.36]{W}, we may
find in $M_{1}[G]$ a sequence $\langle \tau_{\alpha} \mid \alpha < \omega_{1} \rangle$ consisting of
$\pmax$-names in $M_{0}$ such that $\tau_{\alpha,G} = p_{\alpha}$ for all $\alpha < \omega_{1}$.

Via a pre-well-ordering $R$ of length $\gamma$ in $M_{0}$, we may assume that each
$\tau_{\alpha}$ is coded by $R$ and a set of reals $S_{\alpha}$ in $M_{0}$, in such a way that
the sequence $\langle S_{\alpha} \mid \alpha < \omega_{1} \rangle$ is in $M_{1}[G]$. Letting $\eta < \Theta^{M_{0}}$ be
a bound on the Wadge ranks of the sets $S_{\alpha}$, we have that the sequence $\langle
S_{\alpha} \mid \alpha < \omega_{1} \rangle$ is coded by a single set of reals $E$ in $M_{0}$ and an
$\omega_{1}$-sequence of reals in $M_{1}[G]$. Finally, since
$\mathcal{P}(\omega_{1})^{M_{1}[G]} \subseteq \mathrm{L}(\mathbb{R})[G] \subseteq M_{0}[G]$
(as mentioned in the first paragraph of this proof), we have
that $\langle p_{\alpha} \mid \alpha < \omega_{1} \rangle \in M_{0}[G]$.
\end{proof}

The following principle is introduced in Woodin \cite[\S 9.5]{W}. If one assumes that $I$ is the
nonstationary ideal,
then one gets Todorcevic's reflection principle $\SRP(\kappa)$. The principle $\SRP(\omega_{2})$
follows easily from $\MM^{++}(\mathfrak{c})$.

\begin{definition}
Given a cardinal $\kappa \geq \omega_{2}$, $\SRP^{*}(\kappa)$ is the statement that there is a proper,
normal, fine ideal $I \subseteq \mathcal{P}([\kappa]^{\aleph_{0}})$ such that for all stationary
$T \subseteq\omega_{1}$,
$$
\{ X \in [\kappa]^{\aleph_{0}} \mid X \cap \omega_{1} \in T\} \not\in I,
$$
and such that for all $S \subseteq [\kappa]^{\aleph_{0}}$, if $S$ is such that for all stationary $T
\subseteq \omega_{1}$,
$$
\{ X \in S \mid X \cap \omega_{1} \in T\} \not\in I,
$$
then there a set $Y \subseteq \kappa$ such that
\begin{itemize}
\item
 $\omega_{1} \subseteq Y$;
\item
$|Y| = \aleph_{1}$;
\item
 $\cof(\sup(Y)) = \omega_{1}$;
\item
 $S \cap [Y]^{\aleph_{0}}$ contains a club in $[Y]^{\aleph_{0}}$.
\end{itemize}
\end{definition}

The hypothesis of Theorem \ref{theorem:vanilla} is stronger than that of Theorem
\ref{theorem:sqcofoo}, but still below a Woodin limit of Woodin cardinals in terms of consistency strength \cite{SarTra}.
By Woodin \cite[Theorem 9.10]{W}, the hypotheses of the theorem imply that
$\AD^{+}$ holds in $M_{0}$.

\begin{theorem} \label{theorem:vanilla}
Suppose that $M_{0} \subseteq M_{1}$ are models of $\ZF + \AD^{+} + \Unif$ with the same reals such that, letting
$\Gamma_{0} = \mathcal{P}(\mathbb{R}) \cap M_{0}$, the following hold:
\begin{itemize}
\item
 $M_{0} = \HOD^{M_{1}}_{\Gamma_{0}}$;
\item
 $M_{0} \models ``\Theta \text{ is regular}"$;
\item
 $\Theta^{M_{0}} < \Theta^{M_{1}}$;
\item
 $\Theta^{M_{0}}$ has cofinality at least $\omega_{2}$ in $M_{1}$.
\end{itemize}
Let $G \subset \pmax$ be $M_{1}$-generic, and let $H \subset \Add(\omega_{3},1)^{M_{0}[G]}$ be
$M_{1}[G]$-generic.
Then the following hold in $M_{0}[G][H]$:
\begin{itemize}
\item
 $\omega_3$ is threadable; in fact, we have $\neg \square(\omega_{3},\omega)$;
\item
 $\SRP^{*}(\omega_{3})$;
\end{itemize}
\end{theorem}

\begin{proof}
Suppose that $\tau$ is a $\pmax * \Add(\omega_{3},1)$-name in $M_{0}$ for a
$\square(\omega_{3},\omega)$-sequence. We may assume that the realization of $\tau$ comes with
an indexing of each member of the sequence in order type at most $\omega$. In $M_{0}$, $\tau$ is
ordinal definable from some set of reals
$S$.

By Woodin \cite[Theorem 9.39]{W}, $\MM^{++}(\mathfrak{c})$ holds in $M_{1}[G]$.
As in the first paragraph of the proof of Lemma \ref{lemma:closurelemma}, 
$\Theta^{M_{0}}$ has cofinality $\omega_{2}$ in $M_{1}[G]$.
Forcing with $({<}\,\omega_{2})$-directed closed partial orders
of size at most $\mathfrak{c}$ preserves $\MM^{++}(\mathfrak{c})$ (see Larson \cite{L}). It follows then
from Lemma \ref{lemma:closurelemma}
that $\MM^{++} (\mathfrak{c})$ holds in the $\Add(\omega_{3},1)^{M_{0}[G]}$-extension of $M_{1}[G]$,
and thus that in this extension every candidate for a $\square(\Theta^{M_{0}}, \omega)$-sequence is
threaded.

Let $\mathcal{C} = \langle \mathcal{C}_{\alpha} : \alpha < \Theta^{M_{0}} \rangle$ be the realization of
$\tau$ in the $\Add(\omega_{3},1)^{M_{0}[G]}$-extension of $M_{1}[G]$. Since $\Theta^{M_{0}}$ has cofinality at least $\omega_{2}$ in this
extension (which satisfies Choice), $\mathcal{C}$ has at most $\omega$ many
threads, since otherwise one could find a $\mathcal{C}_{\alpha}$ in the sequence with uncountably
many members. Therefore, some member of some $\mathcal{C}_{\alpha}$ in the realization of $\tau$
will be extended by a unique thread through the sequence, and since the realization of $\tau$ indexes
each $\mathcal{C}_{\alpha}$ in order type at most $\omega$, there is in $M_{1}$ a name, ordinal
definable from $S$, for a thread through the realization of $\tau$. This name is then a member of
$M_{0} = \HOD^{M_{1}}_{\Gamma_{0}}$.

That $\SRP^{*}(\omega_{3})$ holds in $M_{0}[G][H]$ follows from the fact that the nonstationary ideal
on $[\Theta^{M_{0}}]^{\aleph_{0}}$ as defined in
$M_{1}[G][H]$ is an element of $M_{0}[G][H]$, and the facts that $|\Theta^{M_{0}}|^{M_{1}[G][H]} =
\aleph_{2}$, and $M_{1}[G][H]$ satisfies $\SRP(\omega_{2})$.
\end{proof}

\begin{question}
Can one improve Theorem \ref{theorem:vanilla} to obtain the failure of
$\square(\omega_{3}, \omega_{1})$?
\end{question}


\end{document}